\documentclass[10pt]{article}
\usepackage{amsthm}
\usepackage{amsmath}
\usepackage{amssymb}
\usepackage{makeidx}

\theoremstyle{definition}
\newtheorem{definition}{Definition}[section]

\newtheorem{example}[definition]{Example}

\newtheorem{remark}[definition]{Remark}

\theoremstyle{plain}
\newtheorem{theorem}[definition]{Theorem}
\newtheorem{lemma}[definition]{Lemma}
\newtheorem{proposition}[definition]{Proposition}
\newtheorem{corollary}[definition]{Corollary}

\numberwithin{equation}{section}

\newcommand{\BI}{\mathbb{BI}}

\newcommand{\BZ}{\mathbb{BZL}}

\def\N{{\mathbb N}}

\def\I{{\mathbb I}}

\def\V{{\mathbb V}}
\def\W{{\mathbb W}}

\begin{document}
\title{Some Properties of Lattice Congruences Preserving Involutions and Their Largest Numbers in the Finite Case}
\author{Claudia MURE\c SAN\thanks{Dedicated to the memory of my beloved grandparents, Gheorghe and Floar\u a--Marioara Mure\c san}\\ 
{\small University of Cagliari, University of Bucharest}\\ 
{\small c.muresan@yahoo.com, cmuresan@fmi.unibuc.ro}}
\date{\today }
\maketitle

\begin{abstract} In this paper, we characterize the congruences of an arbitrary i--lattice, investigate the structure of the lattice they form and how it relates to the structure of the lattice of lattice congruences, then, for an arbitrary non--zero natural number $n$, we determine the largest possible number of congruences of an $n$--element i--lattice, along with the structures of the $n$--element i--lattices with this number of congruences. Our characterizations of the congruences of i--lattices have useful corollaries: determining the congruences of i--chains, the congruence extension property of the variety of distributive i--lattices, a description of the atoms of the congruence lattices of i--lattices, characterizations for the subdirect irreducibility of i--lattices. In terms of the relation between the above--mentioned problem on numbers of congruences of finite i--lattices and its analogue for lattices, while the $n$--element i--lattices with the largest number of congruences turn out to be exactly the $n$--element lattices whose number of congruences is either the largest or the second largest possible, we provide examples of pairs of $n$--element i--lattices and even pseudo--Kleene algebras such that one of them has strictly more congruences, but strictly less lattice congruences than the other.


\end{abstract}

\section{Introduction}
\label{introduction}

In \cite{gcze,kumu}, the largest numbers of congruences of finite lattices are determined, along with the structures of the finite lattices with these numbers of congruences. \cite{gcz} solves the same type of problem for finite semilattices. When we look at the full congruences of lattices with involutions, the so--called i--lattices, we find out that there exist pairs $(L,M)$ of finite i--lattices and even finite pseudo--Kleene algebras such that $L$ and $M$ have the same number of elements and $L$ has strictly more congruences than $M$, but strictly less lattice congruences than $M$, so the solution to the analogue problem for the (involution--preserving) congruences of i--lattices does not derive directly from that on lattice congruences.

In the present paper, we study the congruences of i--lattices, the structure of the lattice they form and the atoms of this lattice, then address the problem above for finite i--lattices, as well as finite algebras from classes involved in the study of quantum logics: the variety of pseudo--Kleene algebras and that of BZ--lattices. In Theorem \ref{maxcgkl} we determine the largest possible number of congruences of an $n$--element i--lattice and the structures of the $n$--element i--lattices with this number of congruences, along with the structure of their congruence lattices. These results produce analogous ones for $n$--element BZ--lattices with the $0$ meet--irreducible, which are thus antiortholattices. In particular cases, we also obtain the second largest possible number of congruences of an $n$--element i--lattice, respectively an $n$--element BZ--lattice with the $0$ meet--irreducible. We exemplify the usefulness of our results on congruences of i--lattices also by deriving from them results on congruences of i--chains, subdirect irreducibility and the congruence extension property for distributive i--lattices.

\section{Preliminaries}
\label{preliminaries}

We shall denote by $\N $ the set of the natural numbers, by $\N ^*=\N \setminus \{0\}$ and, for any $S\subseteq \N $ and any $a,b\in \N $, by $aS+b=\{an+b\ |\ n\in S\}$ and by $aS=aS+0$. For every real number $x$, $\lfloor x\rfloor $ shall be the largest integer not exceeding $x$. The disjoint union of sets shall be denoted $\amalg $. For any non--empty set $M$, we shall denote by $|M|$ the cardinality of $M$, by ${\rm Part}(M)$ the lattice of the partitions of $M$, by ${\rm Eq}(M)$ the lattice of the equivalences of $M$, by $\Delta _M=\{(x,x)\ |\ x\in M\}$ and $\nabla _M=M^2$ and by $eq:{\rm Part}(M)\rightarrow {\rm Eq}(M)$ the canonical lattice isomorphism and; if $n\in \N ^*$ and $\pi =\{M_1,\ldots ,M_n\}\in {\rm Part}(M)$, then $eq(\{M_1,\ldots ,M_n\})$ shall simply be denoted $eq(M_1,\ldots ,M_n)$.

Throughout this paper, all algebras shall be non--empty and, whenever there is no danger of confusion, they will be designated by their underlying sets. By {\em trivial algebra} we mean one--element algebra. Throughout the rest of this section, $A$ will be an arbitrary algebra from a variety $\V $ of algebras of type $\tau $. If $M$ and $N$ are algebras with reducts belonging to $\V $, then we shall denote by $M\cong _{\V }N$ the fact that these reducts of $M$ and $N$ are isomorphic. $({\rm Con}_{\V }(A),\vee ,\cap ,\subseteq ,\Delta _A,\nabla _A)$ shall be the bounded lattice of the congruences of $A$ with respect to the type $\tau $ and, for any $U\subseteq A^2$, $Cg_{A,\V }(U)$ shall be the congruence of $A$ with respect to $\tau $ generated by $U$; for any $a,b\in A$, the principal congruence $Cg_{A,\V }(\{(a,b)\})$ will simply be denoted by $Cg_{A,\V }(a,b)$. Obviously, if $A$ also belongs to a variety $\W $ which has the type $\tau $ or a type that differs from $\tau $ only by a set of constants, then ${\rm Con}_{\V }(A)={\rm Con}_{\W }(A)$ and, for all $U\subseteq A^2$, $Cg_{A,\V }(U)=Cg_{A,\W }(U)$. Clearly, if $B$ is a member of $\V $ and $f:A\rightarrow B$ is an isomorphism in $\V $, then the map $\theta \mapsto f(\theta )=\{(f(a),f(b))\ |\ (a,b)\in \theta \}$ is a lattice isomorphism from ${\rm Con}_{\V }(A)$ to ${\rm Con}_{\V }(B)$ and, for all $U\subseteq A^2$, $f(Cg_{A,\V }(U))=Cg_{B,\V }(f(U))=Cg_{B,\V }(\{(f(a),f(b))\ |\ (a,b)\in U\})$. Also, for any subalgebra $S$ of $A$ and any $\theta \in {\rm Con}_{\V }(A)$, we have $\theta \cap S^2\in {\rm Con}_{\V }(S)$, which is useful for determining the congruences of ordinal and horizontal sums (see below). We will abbreviate by {\em CEP} the {\em congruence extension property}.

\begin{theorem}{\rm \cite[Corollary 2, p. 51]{gralgu}} ${\rm Con}_{\V }(A)$ is a complete sublattice of ${\rm Eq}(A)$.\label{cgeq}\end{theorem}

\begin{corollary} If a reduct of $A$ is a member of a variety $\W $, then ${\rm Con}_{\V }(A)$ is a complete sublattice of ${\rm Con}_{\W }(A)$.\label{cgred}\end{corollary}

If $A$ has an underlying lattice with a $0$, then we will denote by ${\rm Con}_{\V 0}(A)=\{\theta \in {\rm Con}_{\V }(A)\ |\ 0/\theta =\{0\}\}$, which is a complete sublattice of ${\rm Con}_{\V }(A)$ and thus a bounded lattice \cite{rgcmfp}; if $A$ has a bounded lattice reduct, then we shall denote by ${\rm Con}_{\V 01}(A)=\{\theta \in {\rm Con}_{\V }(A)\ |\ 0/\theta =\{0\},1/\theta =\{1\}\}$, which is a complete sublattice of ${\rm Con}_{\V }(A)$ and thus a bounded lattice \cite{rgcmfp}; these results follow routinely from Theorem \ref{cgeq}. In the particular case when $\V $ is the variety of lattices, the index $\V $ shall be eliminated from these notations, as well as those in the paragraph preceeding Theorem \ref{cgeq}.

Recall that, if $(L,\leq ^L)$ is a lattice with greatest element $1^L$ and $(M,\leq ^M)$ is a lattice with smallest element $0^M$, then the {\em ordinal sum} of $L$ with $M$ is the lattice $(L\oplus M,\leq )$ defined by identifying $1^L=0^M=c$ and setting $L\oplus M=(L\setminus \{1^L\})\amalg \{c\}\amalg (M\setminus \{0^M\})$ and $\leq =\leq ^L\cup \leq ^M$. For any $\alpha \in {\rm Con}(L)$ and any $\beta \in {\rm Con}(M)$, we shall denote by $\alpha \oplus \beta =eq((L/\alpha \setminus \{c/\alpha \})\cup \{c/\alpha \cup c/\beta \}\cup (M/\beta \setminus \{c/\beta \}))$. Note that the map $(\alpha ,\beta )\mapsto \alpha \oplus \beta $ sets a lattice isomorphism between ${\rm Con}(L)\times {\rm Con}(M)$ and ${\rm Con}(L\oplus M)$ \cite{eunoucard}. Clearly, the ordinal sum of bounded lattices is an associative operation, and so is the operation $\oplus $ on congruences defined above.

If $(L,\leq ^L,0^L,1^L)$ and $(M,\leq ^M,0^M,1^M)$ are non--trivial bounded lattices, then the {\em horizontal sum} of $L$ with $M$ is the bounded lattice $(L\boxplus M,\leq ,0,1)$ defined by identifying $0^L=0^M=0$ and $1^L=1^M=1$ and setting $L\boxplus M=(L\setminus \{0^L,1^L\})\amalg \{0,1\}\amalg (M\setminus \{0^M,1^M\})$ and $\leq =\leq ^L\cup \leq ^M$. For any $\delta \in {\rm Eq}(L)\setminus \{\nabla _L\}$ and any $\varepsilon \in {\rm Eq}(M)\setminus \{\nabla _M\}$, we shall denote by $\delta \boxplus \varepsilon =eq((L/\delta \setminus \{0/\delta ,1/\delta \})\cup \{0/\delta \cup 0/\varepsilon ,1/\delta \cup 1/\varepsilon \}\cup (M/\varepsilon \setminus \{0/\varepsilon ,1/\varepsilon \}))\in {\rm Eq}(L\boxplus M)$. Clearly, the horizontal sum of non--trivial bounded lattices is a commutative and associative operation, and so is the operation $\boxplus $ on equivalences, which does not always produce congruences when applied to congruences.

Unless we specify otherwise, we shall denote the operations and order relation of a (bounded) lattice in the usual way; the cover relation of a lattice shall be denoted by $\prec $. For any $n\in \N ^*$, we shall denote by ${\cal L}_n$ the $n$--element chain; note that, in \cite{rgcmfp}, it has been denoted by $D_n$ instead of ${\cal L}_n$. We have ${\rm Con}({\cal L}_n)\cong {\cal L}_2^{n-1}$, because ${\rm Con}({\cal L}_1)\cong {\cal L}_1\cong {\cal L}_2^0$ and, if $n\geq 2$, then $\displaystyle {\cal L}_n\cong \bigoplus _{i=1}^{n-1}{\cal L}_2$, hence ${\rm Con}({\cal L}_n)\cong ({\rm Con}({\cal L}_2))^{n-1}\cong {\cal L}_2^{n-1}$. We shall denote by $M_3$ the five--element modular non--distributive lattice and by $N_5$ the five--element non--modular lattice.

Following \cite{kumu}, we shall denote by ${\rm At}(L)$ the set of atoms of a lattice $L$ with smallest element. Now let $L$ be an arbitrary lattice. Then $L^d$ will denote the dual of $L$, so that, clearly, ${\rm Con}(L)={\rm Con}(L^d)$. Let $u\in L$; if $u$ is strictly meet--irreducible in $L$, then we will denote by $u^{+}$ the unique successor of $u$ in $L$, while, if $u$ is strictly join--irreducible, then $u^{-}$ will denote its unique predecessor. If $x,y\in L$, then $[x)$ and $(x]$ shall denote the principal filter and the principal ideal of $L$ generated by $x$, and $[x,y]=[x)\cap (y]$ shall be the interval of $L$ bounded by $x$ and $y$; when $x,y\in \N $ and $L$ is not otherwise specified, $L$ will be the lattice $(\N ,\leq )$, where $\leq $ is the natural order. Recall that the interval $[x,y]$ is called a {\em narrows} iff $x\prec y$, $x$ is meet--irreducible and $y$ is join--irreducible in $L$, which, in the case when $L$ is finite, means that $x$ is strictly meet--irreducible, $y$ is strictly join--irreducible, $x^{+}=y$ and $y^{-}=x$. We will denote by ${\rm Nrw}(L)$ the set of the narrows of $L$. Also, ${\rm Cvx}(L)$ will denote the set of the non--empty convex subsets of $L$.

\section{The Algebras We Work With}

We shall use the terminology and notations from \cite{rgcmfp} in what follows. Note that, in \cite{kal}, Kleene lattices are called {\em bounded normal i--lattices}.

\begin{definition} We call a {\em lattice with involution} or {\em involution lattice} (in brief, {\em i--lattice}) an algebra $(L,\vee ,\wedge ,\leq ,\cdot ^{\prime })$, where $(L,\vee ,\wedge ,\leq )$ is a lattice and $\cdot ^{\prime }$ is a unary operation on $L$, called {\em involution}, that fulfills, for all $a,b\in L$: $a^{\prime \prime }=a$ and, if $a\leq b$, then $b^{\prime }\leq a^{\prime }$. For any i--lattice $L$, we consider the following condition:

$K(L):\quad $ for all $a,b\in L$, $a\wedge a^{\prime }\leq b\vee b^{\prime }$.

A {\em bounded involution lattice} (in brief, {\em bi--lattice}) is an algebra $(L,\vee ,\wedge ,\leq ,\cdot ^{\prime },0,1)$, where $(L,\vee ,\wedge ,\leq ,0,1)$ is a bounded lattice and $(L,\vee ,\wedge ,\leq ,\cdot ^{\prime })$ is an i--lattice. Distributive bi--lattices are called {\em De Morgan algebras}.

A bi--lattice $L$ fulfilling $K(L)$ is called {\em pseudo--Kleene algebra}. Distributive pseudo--Kleene algebras are called {\em Kleene algebras} or {\em Kleene lattices}.

A bi--lattice $L$ is said to be {\em paraorthomodular} iff, for all $a,b\in L$, if $a\leq b$ and $a^{\prime}\wedge b=0$, then $a=b$.

A {\em Brouwer--Zadeh lattice} (in brief, {\em BZ--lattice}) is an algebra $(L,\vee ,\wedge ,\leq ,\cdot ^{\prime },\cdot ^{\sim },0,1)$, where $(L,\vee ,\wedge ,\leq ,\cdot ^{\prime },0,1)$ is a pseudo--Kleene algebra and $\cdot ^{\sim }$ is a unary operation on $L$, called {\em Brouwer complement}, such that, for all $a,b\in L$, we have $a\wedge a^{\sim }=0$, $a\leq a^{\sim \sim }$, $a^{\sim \prime }=a^{\sim \sim }$ and, if $a\leq b$, then $b^{\sim }\leq a^{\sim }$.\label{thealg}\end{definition}

Recall that {\em orthomodular lattices} are exactly the paraorthomodular BZ--lattices whose Kleene complement coincides to their Brouwer complement, while {\em antiortholattices} are exactly the pseudo--Kleene algebras $L$ with the property that $\{a\in L\ |\ a\wedge a^{\prime }=0\}=\{0,1\}$, endowed with the following Brouwer complement, called the {\em trivial Brouwer complement}: $0^{\sim }=1$ and $a^{\sim }=0$ for all $a\in L\setminus \{0\}$, which makes them paraorthomodular BZ--lattices.

We will denote by $\I $ the variety of i--lattices. The variety of bi--lattices is denoted by $\BI $ and that of Brouwer--Zadeh lattices by $\BZ $. Note that antiortholattices do not form a variety.

Obviously, any i--lattice $L$ is self--dual, with $\cdot ^{\prime }:L\rightarrow L$ a dual lattice isomorphism, thus, clearly, $\cdot ^{\prime }:L\rightarrow L^d$ is an i--lattice isomorphism between the i--lattices $L$ and $L^d$ endowed with the same involution $\cdot ^{\prime }$.

If $L$ is an i--lattice, then, for any $M\subseteq L$ and any $U\subseteq L^2$, we shall denote by $M^{\prime }=\{x^{\prime }\ |\ x\in M\}$ and by $U^{\prime }=\{(a^{\prime },b^{\prime })\ |\ (a,b)\in U\}$. Then, for any subset $U$ of $L$ or $L^2$, $U\subseteq U^{\prime }$ iff $U^{\prime }\subseteq U^{\prime \prime }=U$ iff $U=U^{\prime }$.

Let $M$ be a bounded lattice, $f:M\rightarrow M^d$ be a dual lattice isomorphism, and $K$ be a bi--lattice, with the involution $\cdot ^{\prime K}$. If we define $\cdot ^{\prime }:M\oplus K\oplus M^d\rightarrow M\oplus K\oplus M^d$ by: $x^{\prime }=\begin{cases}f(x), & \mbox{for all }x\in M,\\ x^{\prime K}, & \mbox{for all }x\in K,\\ f^{-1}(x), & \mbox{for all }x\in M^d,\end{cases}$ then, clearly, $M\oplus K\oplus M^d$ becomes a bi--lattice, with the involution $\cdot ^{\prime }$; we shall always refer to this structure for a bi--lattice of this form. Note that, if $K$ is a pseudo--Kleene algebra, then $M\oplus K\oplus M^d$ becomes a pseudo--Kleene algebra.

We shall often use the remarks in this paper without referencing them.

\begin{remark} Any i--chain fulfills $K(L)$, thus any bi--chain is a Kleene algebra. Moreover, any bi--chain becomes an antiortholattice with the trivial Brouwer complement. Furthermore, by {\cite{rgcmfp}}, $M$ is a non--trivial bounded lattice and $K$ is a pseudo--Kleene algebra, then $M\oplus K\oplus M^d$ becomes an antiortholattice with the trivial Brouwer complement $\cdot ^{\sim }:M\oplus K\oplus M^d\rightarrow M\oplus K\oplus M^d$.\label{theconstr}\end{remark}

\begin{lemma} Any BZ--lattice with $0$ meet--irreducible (equivalently, with $1$ join--irreducible) is an antiortholattice. In particular, any BZ--chain is an antiortholattice.\label{01irred}\end{lemma}

\begin{proof} Let $L$ be a BZ--lattice with $0$ meet--irreducible and let $a\in L$. Then $a\wedge a^{\sim }=0$, hence $a=0$ or $a^{\sim }=0$, thus the Brouwer complement of $L$ is the trivial Brouwer complement. If $a\wedge a^{\prime}=0$, then $a=0$ or $a^{\prime}=0$, hence $a\in \{0,1\}$. Therefore $L$ is an antiortholattice.\end{proof}

Let $L$ and $M$ be non--trivial bi--lattices, with involutions $\cdot ^{\prime L}$ and $\cdot ^{\prime M}$, respectively. Then $\cdot ^{\prime }:L\boxplus M\rightarrow L\boxplus M$, defined by: $x^{\prime }=\begin{cases}x^{\prime L}, & \mbox{for all }x\in L,\\ x^{\prime M}, & \mbox{for all }x\in M,\end{cases}$ is clearly an involution. Unless mentioned otherwise, any horizontal sum of non--trivial bi--lattices will be organized as a bi--lattice in this way. Hence, for instance, ${\cal L}_2^2\cong {\cal L}_3\boxplus {\cal L}_3$, but ${\cal L}_2^2\ncong _{\BI }{\cal L}_3\boxplus {\cal L}_3$, where by the bi--lattice ${\cal L}_2^2$ we mean, of course, the direct product of bi--lattices ${\cal L}_2\times {\cal L}_2$.

\section{Congruences of i--lattices}

\begin{remark} Let $\V $ be a variety. For any member $M$ of $\V $, any $n\in \N $ and any $x_1,\ldots ,x_n\in M$, we denote by ${\rm Con}_{\V ,x_1,\ldots ,x_n}(M)=\{\theta \in {\rm Con}_{\V }(M)\ |\ (\forall \, i\in [1,n])\, (x_i/\theta =\{a_i\})\}$. Then it is easy to derive from Theorem \ref{cgeq} that ${\rm Con}_{\V ,x_1,\ldots ,x_n}(M)$ is a sublattice of ${\rm Con}_{\V }(M)$ and a bounded lattice.

It is routine to prove that, if $A$ and $B$ are members of $\V $ such that $A\times B$ has no skew congruences, then, for any $n\in \N $, any $a_1,\ldots ,a_n\in A$ and any $b_1,\ldots ,b_n\in B$, we have ${\rm Con}_{\V ,(a_1,b_1),\ldots ,(a_n,b_n)}(A\times B)=\{\alpha \times \beta \ |\ \alpha \in {\rm Con}_{\V ,a_1,\ldots ,a_n}(A),\beta \in {\rm Con}_{\V ,b_1,\ldots ,b_n}(B)\}\cong {\rm Con}_{\V ,a_1,\ldots ,a_n}(A)\times {\rm Con}_{\V ,b_1,\ldots ,b_n}(B)$.\end{remark}

The algebras in Definition \ref{thealg} have underlying lattices, hence they are congruence--dis\-tri\-bu\-tive by Corollary \ref{cgred} and thus the varieties they form have no skew congruences by \cite{bj} and \cite[Theorem $8.5$, p. $85$]{fremck}; hence, for any $\V \in \{\BI ,\BZ \}$ and any $L,M\in \V $, we have ${\rm Con}_{\V 01}(L\times M)=\{\alpha \times \beta \ |\ \alpha \in {\rm Con}_{\V 01}(L),\beta \in {\rm Con}_{\V 01}(M)\}\cong {\rm Con}_{\V 01}(L)\times {\rm Con}_{\V 01}(M)$ by the previous remark. Of course, if $\W $ is a variety of algebras with bounded lattice reducts and $L$ is a member of $\W $ whose underlying bounded lattice is self--dual, in particular if the algebras from $\W $ have bi--lattice reducts, in particular if $\W \in \{\BI ,\BZ \}$, then ${\rm Con}_{\W 01}(L)={\rm Con}_{\W 0}(L)$.

By Corollary \ref{cgred}, for any i--lattice $L$, ${\rm Con}_{\I }(L)$ is a complete sublattice of ${\rm Con}(L)$, so $L$ is subdirectly irreducible if its lattice reduct is, and, for any BZ--lattice $L$, ${\rm Con}_{\BZ }(L)$ is a complete sublattice of ${\rm Con}_{\I }(L)$, so $L$ is subdirectly irreducible if its bi--lattice reduct is.

\begin{remark} For any i--lattice $L$, since $\cdot ^{\prime }$ is a dual lattice automorphism of $L$, it follows that the map $\theta \mapsto \theta ^{\prime }$ is a lattice automorphism of ${\rm Con}(L)$, hence this map is a bijection from ${\rm At}({\rm Con}(L))$ to itself and, for all $\alpha ,\beta \in {\rm Con}(L)$ and all $U\subseteq L^2$, $(\alpha \cap \beta )^{\prime }=\alpha ^{\prime }\cap \beta ^{\prime }$, $(\alpha \vee \beta )^{\prime }=\alpha ^{\prime }\vee \beta ^{\prime }$ and $(Cg_L(U))^{\prime }=Cg_L(U^{\prime })$; in particular, for all $a,b\in L$, $(Cg_L(a,b))^{\prime }=Cg_L(a^{\prime },b^{\prime })$.\label{primecglat}\end{remark}

\begin{remark} For any i--lattice $L$, any $\varepsilon \in {\rm Eq}(L)$ and any $a,b\in L$, we have $(a,b)\in \varepsilon $ iff $(a^{\prime },b^{\prime })\in \varepsilon ^{\prime }$, hence $a^{\prime }/\varepsilon ^{\prime }=\{u^{\prime }\ |\ u\in a/\varepsilon \}=(a/\varepsilon )^{\prime }$, therefore: $\varepsilon =\varepsilon ^{\prime }$ iff, for all $x\in L$, $x/\varepsilon =x/\varepsilon ^{\prime }$, iff, for all $x\in L$, $x^{\prime }/\varepsilon =x^{\prime }/\varepsilon ^{\prime }$, iff, for all $x\in L$, $x^{\prime }/\varepsilon =(x/\varepsilon )^{\prime }$.\label{cgprime}\end{remark}

\begin{lemma} For any i--lattice $L$, ${\rm Con}_{\I }(L)=\{\theta \in {\rm Con}(L)\ |\ \theta =\theta ^{\prime }\}=\{\theta \vee \theta ^{\prime }\ |\ \theta \in {\rm Con}(L)\}=\{\theta \cap \theta ^{\prime }\ |\ \theta \in {\rm Con}(L)\}=\{\theta \in {\rm Con}(L)\ |\ (\forall \, x\in L)\, ((x/\theta )^{\prime }=x^{\prime }/\theta )\}=\{\theta \in {\rm Con}(L)\ |\ (\forall \, \gamma \in L/\theta )\, (\gamma ^{\prime }\in L/\theta )\}$.\label{cginvlat}\end{lemma}

\begin{proof} Let $\theta \in {\rm Con}(L)$. Then $\theta $ preserves the involution of $L$ iff, for all $(a,b)\in \theta $, it follows that $(a^{\prime },b^{\prime })\in \theta $, iff $\theta \subseteq \theta ^{\prime }$ iff $\theta =\theta ^{\prime }$, hence the first equality. The second and third equalities follow from the fact that $(\theta \vee \theta ^{\prime })^{\prime }=\theta \vee \theta ^{\prime }$, $(\theta \cap \theta ^{\prime })^{\prime }=\theta \cap \theta ^{\prime }$ and, if $\theta =\theta ^{\prime }$, then $\theta \vee \theta ^{\prime }=\theta \cap \theta ^{\prime }=\theta $. The fourth and fifth equalities follow from the first and Remark \ref{cgprime}.\end{proof}

\begin{remark} By Remark \ref{primecglat} and Lemma \ref{cginvlat}, for any i--lattice $L$, the map $\theta \mapsto \theta ^{\prime }$ is a lattice automorphism of ${\rm Con}_{\I }(L)$, hence it is a bijection from ${\rm At}({\rm Con}_{\I }(L))$ to itself and, for all all $U\subseteq L^2$ and all $a,b\in L$, $(Cg_{L,\I }(U))^{\prime }=Cg_{L,\I }(U^{\prime })$, in particular $(Cg_{L,\I }(a,b))^{\prime }=Cg_{L,\I }(a^{\prime },b^{\prime })$.\end{remark}

\begin{proposition} If $L$ is an i--lattice and $\theta \in {\rm Con}(L)$ is such that ${\rm Con}(L)=(\theta ]\cup [\theta )$, then $\theta \in {\rm Con}_{\I }(L)$.\label{sureinvcg}\end{proposition}

\begin{proof} $\theta ^{\prime }\in {\rm Con}(L)=(\theta ]\cup [\theta )$, so $\theta =\theta ^{\prime }$, thus $\theta \in {\rm Con}_{\I }(L)$ by Lemma \ref{cginvlat}.\end{proof}

\begin{corollary} Let $L$ be an i--lattice. If ${\rm Con}(L)$ is a chain, then ${\rm Con}_{\I }(L)={\rm Con}(L)$.\end{corollary}

\begin{proposition} Let $L$ be an i--lattice. If ${\rm Con}(L)$ is a Boolean algebra, then ${\rm Con}_{\I }(L)$ is a Boolean subalgebra of ${\rm Con}(L)$.\label{cglatbool}\end{proposition}

\begin{proof} Assume that ${\rm Con}(L)$ is a Boolean algebra, and let $\alpha \in {\rm Con}_{\I }(L)\subseteq {\rm Con}(L)$. Then, for some $\beta \in {\rm Con}(L)$, we have $\alpha \vee \beta =\nabla _L$ and $\alpha \cap \beta =\Delta _L$. Then $\alpha =\alpha ^{\prime }$ and $\beta \vee \beta ^{\prime }\in {\rm Con}_{\I }(L)$ by Lemma \ref{cginvlat}, $\alpha \vee \beta \vee \beta ^{\prime }=\nabla _L$ and $\alpha \cap (\beta \vee \beta ^{\prime })=(\alpha \cap \beta )\vee (\alpha \cap \beta ^{\prime })=\Delta _L\vee (\alpha ^{\prime }\cap \beta ^{\prime })=(\alpha \vee \beta )^{\prime }=\Delta _L^{\prime }=\Delta _L$, hence $\beta \vee \beta ^{\prime }$ is a complement of $\alpha $ in ${\rm Con}_{\I }(L)$, thus also in ${\rm Con}(L)$, so $\beta \vee \beta ^{\prime }=\beta $ by the uniqueness of the complement, that is $\beta ^{\prime }\subseteq \beta $, so $\beta =\beta ^{\prime }$, hence $\beta \in {\rm Con}_{\I }(L)$, again by Lemma \ref{cginvlat}. Therefore ${\rm Con}_{\I }(L)$ is a Boolean subalgebra of ${\rm Con}(L)$.\end{proof}

\begin{corollary} If $L$ is a finite modular i--lattice, in particular if $L$ is a finite De Morgan algebra, in particular if $L$ is a finite Kleene algebra, then ${\rm Con}_{\BI }(L)$ is a Boolean algebra, in particular its cardinality is a natural power of $2$.\label{casecglatbool}\end{corollary}

\begin{proof} By Proposition \ref{cglatbool} and the well--known fact that, if $L$ is a finite modular lattice, then ${\rm Con}(L)$ is a Boolean algebra \cite{gratzer}.\end{proof}

By \cite[Theorem $3.2$]{hps}, the variety of De Morgan algebras has the CEP. Let us use the known fact that the variety of distributive lattices has the CEP to prove that boundeness is not necessary in this result.

\begin{proposition} The variety of distributive i--lattices has the CEP.\end{proposition}

\begin{proof} Let $L$ be a distributive i--lattice, $S$ be an i--sublattice of $L$ and $\sigma \in {\rm Con}_{\I }(S)\subseteq {\rm Con}(S)$. Then there exists a $\theta \in {\rm Con}(L)$ such that $\theta \cap S^2=\sigma $. By Lemma \ref{cginvlat}, $\theta \cap \theta ^{\prime }\in {\rm Con}_{\I }(L)$ and $\theta ^{\prime }\cap S^2=\theta ^{\prime }\cap (S^{\prime })^2=(\theta \cap S^2)^{\prime }=\sigma ^{\prime }=\sigma $, hence $\theta \cap \theta ^{\prime }\cap S^2=\theta \cap S^2\cap \theta ^{\prime }\cap S^2=\sigma $.\end{proof}

\begin{remark} Obviously, distributive i--lattices are not uniquely congruence--extensible; for instance, ${\cal L}_2$ is an i--sublattice of ${\cal L}_2^2$ and $\Delta _{{\cal L}_2^2}\cap {\cal L}_2=(\Delta _{{\cal L}_2}\times \nabla _{{\cal L}_2})\cap {\cal L}_2=(\nabla _{{\cal L}_2}\times \Delta _{{\cal L}_2})\cap {\cal L}_2=\Delta _{{\cal L}_2}$.\end{remark}

\begin{remark}{\rm \cite{eucard}} From the well--known fact that, for any lattice $L$, ${\rm Con}(L)\subseteq \{eq(\pi )\ |\ \pi \in {\rm Part}(L),\pi \subseteq {\rm Cvx}(L)\}$ \cite{gratzer}, it is immediate that, for any chain $C$, ${\rm Con}(C)=\{eq(\pi )\ |\ \pi \in {\rm Part}(C),\pi \subseteq {\rm Cvx}(C)\}$, that is the lattice congruences of any chain are exactly its equivalences with all classes convex.

The only chains which are subdirectly irreducible as lattices are ${\cal L}_1$ and ${\cal L}_2$, because ${\cal L}_1$ is trivial, ${\cal L}_2$ is simple and, if a chain $C$ has $|C|>2$, then there exist elements $a,b,c\in C$ with $a<b<c$, so, if we denote by $\alpha =eq(\{[a,b]\}\cup \{\{x\}\ |\ x\in C\setminus [a,b]\})$ and $\beta =eq(\{[b,c]\}\cup \{\{x\}\ |\ x\in C\setminus [b,c]\})$, then $\alpha ,\beta \in {\rm Con}(C)\setminus \{\Delta _C\}$ and $\alpha \cap \beta =\Delta _C$, so $C$ is subdirectly reducible.\label{cgchain}\end{remark}

For any i--lattice $L$, we denote by $N(L)=\{x\in L\ |\ x<x^{\prime }\}$, $Z(L)=\{x\in L\ |\ x=x^{\prime }\}$, $P(L)=\{x\in L\ |\ x>x^{\prime }\}$, $NZ(L)=N(L)\cup Z(L)$ and $PZ(L)=P(L)\cup Z(L)$. Clearly, if $K(L)$ holds, in particular if $L$ is a chain or a pseudo--Kleene algebra, then $|Z(L)|\leq 1$.

\begin{remark} Let $C$ be an i--chain and $\gamma \in {\rm Cvx}(C)$. Then, clearly, $C=N(C)\cup Z(C)\cup P(C)=NZ(C)\cup PZ(C)$, so that $\gamma =(\gamma \cap NZ(C))\cup (\gamma \cap PZ(C))$. $NZ(C),PZ(C),\gamma ,NZ(C)\setminus \gamma ,PZ(C)\setminus \gamma \in {\rm Cvx}(C)$, thus they are sublattices of $C$ since $C$ is a chain. For any $x\in N(C)$ and any $y\in P(C)$, we have $x<y$ and, if a $z\in Z(C)$ exists, then $x<z<y$. Also, if $\gamma \cap NZ(C)\neq \emptyset \neq \gamma \cap PZ(C)$, then, for any $x\in NZ(C)\setminus \gamma $, any $z\in \gamma $ and any $y\in PZ(C)\setminus \gamma $, we have $x<z<y$. By Remark \ref{cgprime}, for any $\varepsilon \in {\rm Eq}(C)$, we have: $\varepsilon =\varepsilon ^{\prime }$ iff $x^{\prime }/\varepsilon =(x/\varepsilon )^{\prime }$ for all $x\in C$ iff $x^{\prime }/\varepsilon =(x/\varepsilon )^{\prime }$ for all $x\in NZ(C)$ iff $\varepsilon \cap PZ(C)^2=(\varepsilon \cap NZ(C)^2)^{\prime }$.\label{nzp}\end{remark}

We conclude that the (involution--preserving) congruences of any i--chain $C$ are exactly its equivalences $\varepsilon $ with the property that, for any class $\gamma $ of $\varepsilon $, $\gamma ^{\prime }$ is also a class of $\varepsilon $, and, out of these classes, at most one class $\gamma $ satisfies $\gamma =\gamma ^{\prime }$, while the other classes $\gamma $ are disjoint from $\gamma ^{\prime }$; this makes the (involution--preserving) congruences of $C$ be obtained by considering arbitrary equivalences $\varepsilon $ on $C$ with convex classes and taking all their classes $\beta $ included in the negative cone $N(C)$ of $C$, along with the sets $\beta ^{\prime }$, which are convex subsets of the positive cone $P(C)$ and, in the case when a class $\gamma $ of $\varepsilon $ is neither included in $N(C)$, nor in $P(C)$, so that $x<z<y$ for all $x\in N(C)\setminus \gamma $, all $z\in \gamma $ and all $y\in P(C)\setminus \gamma $, also considering the convex subset $\gamma \cup \gamma ^{\prime }$ of $C$; this makes the (involution--preserving) congruences of $C$ be completely determined by the lattice congruences of its subchain $N(C)$:

\begin{lemma} For any i--chain $C$, ${\rm Con}_{\I }(C)=\{eq(\pi )\ |\ \pi \in {\rm Part}(C),\pi \subseteq {\rm Cvx}(C),(\forall \, \gamma \in \pi )\, (\gamma ^{\prime }\in \pi )\}=\{eq(\rho \cup \{\gamma \cup \gamma ^{\prime }\}\cup \{\beta ^{\prime }\ |\ \beta \in \rho \})\ |\ \gamma \in {\rm Cvx}(NZ(C)),(\forall \, x\in NZ(C)\setminus \gamma )\, (\forall \, y\in \gamma )\, (x<y),\rho \in {\rm Part}(NZ(C)\setminus \gamma ),\rho \subseteq {\rm Cvx}(C)\}$.\label{cgichains}\end{lemma}

\begin{proof} By Lemma \ref{cginvlat} and Remarks \ref{cgchain} and \ref{nzp}.\end{proof}

\begin{proposition} For any i--chain $C$, ${\rm Con}_{\I }(C)\cong {\rm Con}(N(C))\times {\cal L}_2$. In particular, if $Z(C)\neq \emptyset $, then ${\rm Con}_{\I }(C)\cong {\rm Con}(NZ(C))$.\label{congichains}\end{proposition}

\begin{proof} By Lemma \ref{cgichains}, the map $\Psi :{\rm Con}_{\I }(C)\rightarrow {\rm Con}(N(C))\times {\cal L}_2$, defined in the following way, for all $\theta \in {\rm Con}_{\I }(C)$:\begin{itemize}
\item $\Psi (\theta )=(\theta \cap N(C)^2,0)$, if every $\gamma \in C/\theta $ is such that $\gamma \subseteq N(C)$ or $\gamma =Z(C)$ or $\gamma \subseteq P(C)$,
\item $\Psi (\theta )=(\theta \cap N(C)^2,1)$, if some $\gamma \in C/\theta $ satisfies $\gamma \setminus N(C)\neq \emptyset \neq \gamma \cap N(C)$,\end{itemize}

\noindent is a lattice isomorphism. In the case when $Z(C)\neq \emptyset $, so that $Z(C)=\{z\}$ for some $z\in C$, the map $\Phi :{\rm Con}(NZ(C))\rightarrow {\rm Con}(N(C))\times {\cal L}_2$, defined in the following way, for all $\zeta \in {\rm Con}(NZ(C))$:\begin{itemize}
\item $\Phi (\zeta )=(\zeta \cap N(C)^2,0)$, if $z/\zeta =\{z\}$,
\item $\Phi (\zeta )=(\zeta \cap N(C)^2,1)$, if $z/\zeta \supsetneq \{z\}$,\end{itemize}

\noindent is a lattice isomorphism.\end{proof}

\begin{corollary} The only subdirectly irreducible i--chains are ${\cal L}_1$, ${\cal L}_2$ and ${\cal L}_3$.\label{siichains}\end{corollary}

\begin{proof} By Proposition \ref{congichains} and the fact that, whenever $|C|\leq 3$, we have $|NZ(C)|\leq 2$, so that the chain $NZ(C)$ is subdirectly irreducible as a lattice and the fact that, whenever $|C|\geq 4$, we have $|N(C)|\geq 2$, so $\Delta _{N(C)}\neq \nabla _{N(C)}$, thus the elements $(\Delta _{N(C)},1)$ and $(\nabla _{N(C)},0)$ of ${\rm Con}(N(C))\times {\cal L}_2$ are nonzero and their meet is $0$, hence ${\rm Con}_{\I }(C)\cong {\rm Con}(N(C))\times {\cal L}_2$ can not have a single atom.\end{proof}

The proofs of the following two lemmas from \cite{rgcmfp} and \cite{kumu} are straightforward. The last statement in the next lemma is the particular finite case for Proposition \ref{congichains} and Corollary \ref{cgaolchains} below.

\begin{lemma}{\rm \cite{rgcmfp}}\begin{itemize}
\item If $M$ is a bounded lattice and $K$ is a bi--lattice, then ${\rm Con}_{\BI }(L)=\{\alpha \oplus \beta \oplus \alpha ^{\prime }\ |\ \alpha \in {\rm Con}(M),\beta \in {\rm Con}_{\BI }(K)\}\cong {\rm Con}(M)\times {\rm Con}_{\BI }(K)$.
\item For any non--trivial antiortholattice $L$, ${\rm Con}_{\BZ 01}(L)={\rm Con}_{\BI 01}(L)$ and ${\rm Con}_{\BZ }(L)={\rm Con}_{\BI 01}(L)\cup \{\nabla _L\}\cong {\rm Con}_{\BI 01}(L)\oplus {\cal L}_2$.
\item If $M$ is a non--trivial bounded lattice, $K$ is a pseudo--Kleene algebra and $L=M\oplus K\oplus M^d$, then ${\rm Con}_{\BZ }(L)=\{\alpha \oplus \beta \oplus \alpha ^{\prime }\ |\ \alpha \in {\rm Con}_0(M),\beta \in {\rm Con}_{\BI }(K)\}\cup \{\nabla _L\}\cong ({\rm Con}_0(M)\times {\rm Con}_{\BI }(K))\oplus {\cal L}_2$.
\item For any $n\in \N ^*$, ${\rm Con}_{\BI }({\cal L}_n)\cong {\cal L}_2^{\lfloor n/2\rfloor}$ and, if $n\geq 2$, then ${\rm Con}_{\BZ }({\cal L}_n)\cong {\cal L}_2^{\lfloor n/2\rfloor -1}\oplus {\cal L}_2$.
\end{itemize}\label{thecg}\end{lemma}

\begin{corollary} For any antiortholattice chain $C$, ${\rm Con}_{\BZ 01}(C)={\rm Con}_{\BI 01}(C)=\{eq(\pi )\ |\ \pi \in {\rm Part}(C),\pi \subseteq {\rm Cvx}(C),\{0\}\in \pi ,\{1\}\in \pi ,(\forall \, \gamma \in \pi )\, (\gamma ^{\prime }\in \pi )\}=\{eq(\{\{0\},\{1\},\gamma \cup \gamma ^{\prime }\}\cup \rho \cup \{\beta ^{\prime }\ |\ \beta \in \rho \})\ |\ \gamma \in {\rm Cvx}(NZ(C)\setminus \{0\}),(\forall \, x\in NZ(C)\setminus \gamma )\, (\forall \, y\in \gamma )\, (x<y),\rho \in {\rm Part}(NZ(C)\setminus (\{0\}\cup \gamma )),\rho \subseteq {\rm Cvx}(C)\}\cong {\rm Con}_{\BI }(C\setminus \{0,1\})\cong {\rm Con}(N(C))$, so that ${\rm Con}_{\BZ }(C)\cong {\rm Con}(N(C))\oplus {\cal L}_2$ and thus $C$ is subdirectly irreducible iff $|C|\leq 5$.\label{cgaolchains}\end{corollary}

\begin{proof} By Remark \ref{cgchain}, for any bounded chain $C$, we have ${\rm Con}_{01}(C)=\{eq(\pi )\ |\ \pi \in {\rm Part}(C),\pi \subseteq {\rm Cvx}(C),\{0\}\in \pi ,\{1\}\in \pi \}$. The equalities and isomorphisms in the enunciation now follow by Lemmas \ref{thecg} and \ref{cgichains}, along with Proposition \ref{congichains} and Corollary \ref{siichains}; indeed, Proposition \ref{congichains} shows that ${\rm Con}_{\BI 01}(C)\cong {\rm Con}(N(C)\setminus \{0\})\times {\cal L}_2\cong {\rm Con}(N(C))$, with the latter isomorphism established similarly to $\Phi $ in its proof.\end{proof}

\begin{remark} Clearly, if $M$ is a lattice with $0$, then ${\rm Con}_0({\cal L}_2\oplus M)\cong {\rm Con}(M)$, because the map $\alpha \mapsto \alpha \cap M^2$ sets a lattice isomorphism from ${\rm Con}_0({\cal L}_2\oplus M)$ to ${\rm Con}(M)$.

Similarly, if $L$ is a bi--lattice with $0$ strictly meet--irreducible, so that $1$ is strictly join--irreducible by the self--duality of $L$, then $L={\cal L}_2\oplus K\oplus {\cal L}_2$ for some bi--lattice $K$, ${\rm Con}_{01}(L)=\{\Delta _{{\cal L}_2}\oplus \alpha \oplus \Delta _{{\cal L}_2}\ |\ \alpha \in {\rm Con}(K)\}=\{eq(\{\{0\},\{1\}\}\cup K/\alpha )\ |\ \alpha \in {\rm Con}(K)\}\cong {\rm Con}(K)$ and ${\rm Con}_{\BI 01}(L)=\{\Delta _{{\cal L}_2}\oplus \alpha \oplus \Delta _{{\cal L}_2}\ |\ \alpha \in {\rm Con}_{\BI }(K)\}=\{eq(\{\{0\},\{1\}\}\cup K/\alpha )\ |\ \alpha \in {\rm Con}_{\BI }(K)\}\cong {\rm Con}_{\BI }(K)$.

By Remark \ref{theconstr} and Lemma \ref{thecg}, if, furthermore, $K$ is a pseudo--Kleene algebra, then $L={\cal L}_2\oplus K\oplus {\cal L}_2$ is an antiortholattice and ${\rm Con}_{\BZ }(L)={\rm Con}_{\BI 01}(L)\cup \{\nabla _L\}\cong {\rm Con}_{\BI 01}(L)\oplus {\cal L}_2\cong {\rm Con}_{\BI }(K)\oplus {\cal L}_2$, therefore the antiortholattice $L$ is subdirectly irreducible iff the bi--lattice $K$ is subdirectly irreducible.\label{lksi}\end{remark}

\begin{lemma}{\rm \cite{kumu}} For any bounded lattices $H$ and $K$ with $|H|>2$ and $|K|>2$, we have:\begin{itemize}
\item $\{\alpha \boxplus \beta \ |\ \alpha \in {\rm Con}_{01}(H),\beta \in {\rm Con}_{01}(K)\}\cup \{\nabla _{H\boxplus K}\}\subseteq {\rm Con}(H\boxplus K)\subseteq \{\alpha \boxplus \beta \ |\ \alpha \in {\rm Con}_{01}(H),\beta \in {\rm Con}_{01}(K)\}\cup \{eq(H\setminus \{0\},K\setminus \{1\}),eq(H\setminus \{1\},K\setminus \{0\}),\nabla _{H\boxplus K}\}$.\end{itemize}

For all $t\in \N ^*$ and any bounded lattices $L_1,\ldots ,L_t$, if $L=\boxplus _{i=1}^t({\cal L}_2\oplus L_i\oplus {\cal L}_2)$, then:\begin{itemize}
\item if $t=2$, then ${\rm Con}(L)=\{eq(\{\{0\},\{1\}\}\cup L_1/\alpha \cup L_2/\beta )\ |\ \alpha \in {\rm Con}(L_1),\beta \in {\rm Con}(L_2)\}\cup \{eq(\{0\}\cup L_1,L_2\cup \{1\}),eq(\{0\}\cup L_2,L_1\cup \{1\}),\nabla _L\}\cong ({\rm Con}(L_1)\times {\rm Con}(L_2))\oplus {\cal L}_2^2$;
\item if $t\geq 3$, then $\displaystyle {\rm Con}(L)=\{eq(\{\{0\},\{1\}\}\cup \bigcup _{i=1}^tL_i/\alpha _i)\ |\ (\forall \, i\in [1,t])\, (\alpha _i\in {\rm Con}(L_i))\}\cup \{\nabla _L\}\cong (\prod _{i=1}^t{\rm Con}(L_i))\oplus {\cal L}_2$.\end{itemize}\label{fortheex}\end{lemma}

With the notations in the previous lemma, note that, if $\alpha _i\in {\rm Con}(L_i)$ for all $i\in [1,t]$, then $\displaystyle eq(\{\{0\},\{1\}\}\cup \bigcup _{i=1}^tL_i/\alpha _i)=\boxplus _{i=1}^t(\Delta _{{\cal L}_2}\oplus \alpha _i\oplus \Delta _{{\cal L}_2})$.

\begin{proposition} Let $t\in \N \setminus \{0,1\}$, $K_1,\ldots ,K_t,L_1,\ldots ,L_t$ be bi--lattices such that $|K_i|>2$ for all $i\in [1,t]$, $K=\boxplus _{i=1}^tK_i$ and $L=\boxplus _{i=1}^t({\cal L}_2\oplus L_i\oplus {\cal L}_2)$. Then:\begin{itemize}
\item $\displaystyle {\rm Con}_{\BI }(K)=\{\boxplus _{i=1}^t\alpha _i\ |\ (\forall \, i\in [1,t])\, (\alpha _i\in {\rm Con}_{\BI 01}(K_i))\}\cup \{\nabla _K\}\cong (\prod _{i=1}^t{\rm Con}_{\BI 01}(K_i))\oplus {\cal L}_2$;
\item $\displaystyle {\rm Con}_{\BI }(L)=\{eq(\{\{0\},\{1\}\}\cup \bigcup _{i=1}^tL_i/\alpha _i)\ |\ (\forall \, i\in [1,t])\, (\alpha _i\in {\rm Con}_{\BI }(L_i))\}\cup \{\nabla _L\}\cong (\prod _{i=1}^t{\rm Con}_{\BI }(L_i))\oplus {\cal L}_2$.\end{itemize}\label{alsofortheex}\end{proposition}

\begin{proof} By the associativity of the horizontal sum, Lemmas \ref{cginvlat} and \ref{fortheex} and the fact that $eq(K_1\setminus \{0\},K_2\setminus \{1\})^{\prime }=eq(K_1\setminus \{1\},K_2\setminus \{0\})\neq eq(K_1\setminus \{0\},K_2\setminus \{1\})$, hence $eq(K_1\setminus \{0\},K_2\setminus \{1\})\notin {\rm Con}_{\BI }(K_1\boxplus K_2)$.\end{proof}

\begin{remark} Obviously, if $B$ is a Boolean algebra, then $B$ is a Kleene algebra, with its Boolean complement as involution, and has ${\rm Con}_{\BI }(B)={\rm Con}(B)$. Actually, by \cite{bruhar}, all Boolean algebras are orthomodular lattices, which are among the pseudo--Kleene algebras $K$ with the property that ${\rm Con}_{\BI }(K)={\rm Con}(K)$. But the Boolean complement is not necessarily the only involution on $B$, for instance the Boolean algebra ${\cal L}_2^2$ can also be organized as the bi--lattice ${\cal L}_3\boxplus {\cal L}_3$, which is not a pseudo--Kleene algebra, because, with the notations in the leftmost figure below, $a=a\wedge a^{\prime }\nleq b\vee b^{\prime }=b$. Clearly, if $L$ is a bounded lattice, $A$ and $B$ are non--trivial bi--lattices and $A$ is not a pseudo--Kleene algebra, then $A\boxplus B$ is not a pseudo--Kleene algebra, while $L\oplus B\oplus L^d$ is a pseudo--Kleene algebra iff $B$ is a pseudo--Kleene algebra.\vspace*{-17pt}

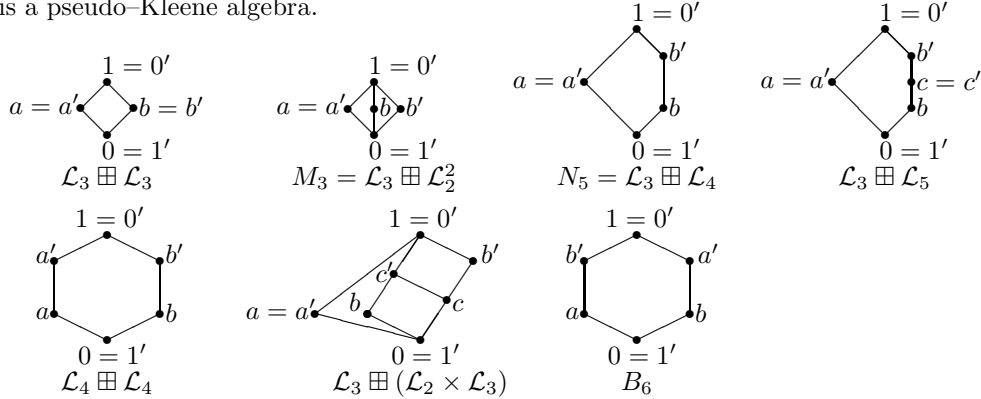
\begin{center}\begin{tabular}{cccc}
\begin{picture}(40,50)(0,0)
\put(20,0){\circle*{3}}
\put(30,10){\circle*{3}}
\put(20,20){\circle*{3}}
\put(10,10){\circle*{3}}
\put(18,-9){$0=1^{\prime }$}
\put(18,23){$1=0^{\prime }$}
\put(-17,8){$a=a^{\prime }$}
\put(32,7){$b=b^{\prime }$}
\put(20,20){\line(-1,-1){10}}
\put(20,20){\line(1,-1){10}}
\put(20,0){\line(-1,1){10}}
\put(20,0){\line(1,1){10}}
\end{picture}
&\hspace*{5pt}
\begin{picture}(40,50)(0,0)
\put(20,0){\circle*{3}}
\put(20,10){\circle*{3}}
\put(30,10){\circle*{3}}
\put(20,20){\circle*{3}}
\put(10,10){\circle*{3}}
\put(18,-9){$0=1^{\prime }$}
\put(18,23){$1=0^{\prime }$}
\put(-17,8){$a=a^{\prime }$}
\put(22,7){$b$}
\put(32,7){$b^{\prime }$}
\put(20,20){\line(-1,-1){10}}
\put(20,20){\line(1,-1){10}}
\put(20,0){\line(-1,1){10}}
\put(20,0){\line(1,1){10}}
\put(20,0){\line(0,1){20}}
\end{picture}
&\hspace*{3pt}
\begin{picture}(40,50)(0,0)
\put(20,0){\circle*{3}}
\put(20,40){\circle*{3}}
\put(0,20){\circle*{3}}
\put(30,10){\circle*{3}}
\put(30,30){\circle*{3}}
\put(18,-9){$0=1^{\prime }$}
\put(18,43){$1=0^{\prime }$}
\put(-27,18){$a=a^{\prime }$}
\put(32,7){$b$}
\put(32,28){$b^{\prime }$}
\put(20,0){\line(-1,1){20}}
\put(20,40){\line(1,-1){10}}
\put(20,0){\line(1,1){10}}
\put(0,20){\line(1,1){20}}
\put(20,40){\line(1,-1){10}}
\put(30,10){\line(0,1){20}}
\end{picture}
&\hspace*{28pt}
\begin{picture}(40,50)(0,0)
\put(20,0){\circle*{3}}
\put(30,20){\circle*{3}}
\put(20,40){\circle*{3}}
\put(0,20){\circle*{3}}
\put(30,10){\circle*{3}}
\put(30,30){\circle*{3}}
\put(18,-9){$0=1^{\prime }$}
\put(18,43){$1=0^{\prime }$}
\put(-27,18){$a=a^{\prime }$}
\put(32,7){$b$}
\put(32,17){$c=c^{\prime }$}
\put(32,28){$b^{\prime }$}
\put(20,0){\line(-1,1){20}}
\put(20,40){\line(1,-1){10}}
\put(20,0){\line(1,1){10}}
\put(0,20){\line(1,1){20}}
\put(20,40){\line(1,-1){10}}
\put(30,10){\line(0,1){20}}
\end{picture} \vspace*{7pt}\\ 
${\cal L}_3\boxplus {\cal L}_3$ &\hspace*{5pt} $M_3={\cal L}_3\boxplus {\cal L}_2^2$ &\hspace*{3pt} $N_5={\cal L}_3\boxplus {\cal L}_4$ &\hspace*{28pt} ${\cal L}_3\boxplus {\cal L}_5$ \vspace*{5pt}\\ 
\begin{picture}(60,50)(0,0)
\put(30,0){\circle*{3}}
\put(30,0){\line(2,1){20}}
\put(30,0){\line(-2,1){20}}
\put(30,40){\line(-2,-1){20}}
\put(30,40){\line(2,-1){20}}
\put(10,30){\line(0,-1){20}}
\put(50,30){\line(0,-1){20}}
\put(30,40){\circle*{3}}
\put(10,10){\circle*{3}}
\put(10,30){\circle*{3}}
\put(50,10){\circle*{3}}
\put(50,30){\circle*{3}}
\put(19,-9){$0=1^{\prime }$}
\put(18,43){$1=0^{\prime }$}
\put(52,7){$b$}
\put(52,29){$b^{\prime }$}
\put(3,8){$a$}
\put(3,29){$a^{\prime }$}
\end{picture}
&\hspace*{40pt}
\begin{picture}(60,50)(0,0)
\put(30,0){\circle*{3}}
\put(30,0){\line(-4,1){40}}
\put(30,40){\line(-4,-3){40}}
\put(40,15){\line(-2,1){20}}
\put(30,0){\line(-2,1){20}}
\put(10,10){\line(2,3){20}}
\put(50,30){\line(-2,-3){20}}
\put(30,40){\line(2,-1){20}}
\put(30,40){\circle*{3}}
\put(-10,10){\circle*{3}}
\put(10,10){\circle*{3}}
\put(10,10){\circle*{3}}
\put(20,25){\circle*{3}}
\put(40,15){\circle*{3}}
\put(50,30){\circle*{3}}
\put(19,-9){$0=1^{\prime }$}
\put(18,43){$1=0^{\prime }$}
\put(53,29){$b^{\prime }$}
\put(-36,8){$a=a^{\prime }$}
\put(3,11){$b$}
\put(42,11){$c$}
\put(13,22){$c^{\prime }$}
\end{picture}
&\hspace*{3pt}
\begin{picture}(60,50)(0,0)
\put(30,0){\circle*{3}}
\put(30,0){\line(2,1){20}}
\put(30,0){\line(-2,1){20}}
\put(30,40){\line(-2,-1){20}}
\put(30,40){\line(2,-1){20}}
\put(10,30){\line(0,-1){20}}
\put(50,30){\line(0,-1){20}}
\put(30,40){\circle*{3}}
\put(10,10){\circle*{3}}
\put(10,30){\circle*{3}}
\put(50,10){\circle*{3}}
\put(50,30){\circle*{3}}
\put(19,-9){$0=1^{\prime }$}
\put(18,43){$1=0^{\prime }$}
\put(52,7){$b$}
\put(53,29){$a^{\prime }$}
\put(3,8){$a$}
\put(3,29){$b^{\prime }$}
\end{picture} & \vspace*{7pt}\\ 
${\cal L}_4\boxplus {\cal L}_4$ &\hspace*{40pt} ${\cal L}_3\boxplus ({\cal L}_2\times {\cal L}_3)$ &\hspace*{3pt} $B_6$ &\end{tabular}\end{center}\vspace*{-6pt}

The self--duality property shows that the only way in which the lattices $M_3$ and $N_5$ can be organized as i--lattices is as the horizontal sum of the bi--lattices ${\cal L}_3$ and ${\cal L}_2^2$, respectively the horizontal sum of the bi--lattices ${\cal L}_3$ and ${\cal L}_4$, which makes $M_3$ a pseudo--Kleene algebra, but, with the notations above, since $a=a\wedge a^{\prime }\nleq b\vee b^{\prime }=b^{\prime }$, it follows that $N_5$ can not be organized as a pseudo--Kleene algebra. A similar argument shows that, for any $k,m\in \N \setminus \{0,1,2\}$, the bi--lattice ${\cal L}_k\boxplus {\cal L}_m$ is not a pseudo--Kleene algebra; moreover, if $k$ is odd and $m$ is even, then the lattice ${\cal L}_k\boxplus {\cal L}_m$ can not be organized as a pseudo--Kleene algebra; more generally, if $M$ and $K$ are bounded lattices such that $K$ is non--trivial, $k,m\in \N \setminus \{0,1,2\}$ such that $k$ is odd and $m$ and $|K|$ are even, and $A$ and $B$ are non--trivial bi--lattices, with involutions $^{\prime A}$ and $^{\prime B}$, respectively, such that $0_{AB}$ and $1_{AB}$ are the smallest and the greatest element of $A\boxplus B$, respectively, and, for some $a\in A$ and $b\in B$, $a\wedge a^{\prime A}\notin \{0_{AB},1_{AB}\}$ and $b\vee b^{\prime A}\neq 1_{AB}$, then the bi--lattice $M\oplus (A\boxplus B)\oplus M^d$ is not a pseudo--Kleene algebra, and the lattices $M\oplus ({\cal L}_k\boxplus {\cal L}_m)\oplus M^d$ and $M\oplus ({\cal L}_k\boxplus K)\oplus M^d$ can not be organized as pseudo--Kleene algebras.

${\rm Con}_{\BI }({\cal L}_2\times {\cal L}_3)\cong {\rm Con}_{\BI }({\cal L}_2)\times {\rm Con}_{\BI }({\cal L}_3)\cong {\cal L}_2^2$ and ${\rm Con}_{\BI 01}({\cal L}_2\times {\cal L}_3)\cong {\rm Con}_{\BI 01}({\cal L}_2)\times {\rm Con}_{\BI 01}({\cal L}_3)\cong {\cal L}_1\cong {\rm Con}_{\BI 01}({\cal L}_2)\times {\rm Con}_{\BI 01}({\cal L}_2)\cong {\rm Con}_{\BI 01}({\cal L}_2^2)$. By Proposition \ref{alsofortheex}, ${\rm Con}_{\BI }(M_3)={\rm Con}_{\BI }({\cal L}_3\boxplus {\cal L}_2^2)\cong ({\cal L}_1\times {\cal L}_1)\oplus {\cal L}_2\cong {\cal L}_2$, ${\rm Con}_{\BI }({\cal L}_3\boxplus {\cal L}_3)={\rm Con}_{\BI }(({\cal L}_2\oplus {\cal L}_1\oplus {\cal L}_2)\boxplus ({\cal L}_2\oplus {\cal L}_1\oplus {\cal L}_2))\cong {\cal L}_2\oplus {\cal L}_1\cong {\cal L}_2$, ${\rm Con}_{\BI }({\cal L}_3\boxplus {\cal L}_3)={\rm Con}_{\BI }(({\cal L}_2\oplus {\cal L}_1\oplus {\cal L}_2)\boxplus ({\cal L}_2\oplus {\cal L}_1\oplus {\cal L}_2))\cong {\cal L}_2\oplus {\cal L}_1\cong {\cal L}_2$, ${\rm Con}_{\BI }({\cal L}_3\boxplus {\cal L}_3)={\rm Con}_{\BI }(({\cal L}_2\oplus {\cal L}_1\oplus {\cal L}_2)\boxplus ({\cal L}_2\oplus {\cal L}_1\oplus {\cal L}_2))\cong {\cal L}_2\oplus {\cal L}_1\cong {\cal L}_2$, ${\rm Con}_{\BI }(N_5)={\rm Con}_{\BI }({\cal L}_3\boxplus {\cal L}_4)={\rm Con}_{\BI }(({\cal L}_2\oplus {\cal L}_1\oplus {\cal L}_2)\boxplus ({\cal L}_2\oplus {\cal L}_2\oplus {\cal L}_2))\cong {\cal L}_2\oplus {\cal L}_2\cong {\cal L}_3$, ${\rm Con}_{\BI }({\cal L}_3\boxplus {\cal L}_5)={\rm Con}_{\BI }(({\cal L}_2\oplus {\cal L}_1\oplus {\cal L}_2)\boxplus ({\cal L}_2\oplus {\cal L}_3\oplus {\cal L}_2))\cong {\cal L}_2\oplus {\cal L}_2\cong {\cal L}_3$, ${\rm Con}_{\BI }({\cal L}_4\boxplus {\cal L}_4)={\rm Con}_{\BI }(({\cal L}_2\oplus {\cal L}_2\oplus {\cal L}_2)\boxplus ({\cal L}_2\oplus {\cal L}_2\oplus {\cal L}_2))\cong {\cal L}_2\oplus {\cal L}_2\cong {\cal L}_3$ and ${\rm Con}_{\BI }({\cal L}_3\boxplus ({\cal L}_2\times {\cal L}_3))\cong ({\cal L}_1\times {\cal L}_1)\oplus {\cal L}_2\cong {\cal L}_2$, thus $|{\rm Con}_{\BI }(N_5)|=|{\rm Con}_{\BI }({\cal L}_3\boxplus {\cal L}_5)|=|{\rm Con}_{\BI }({\cal L}_4\boxplus {\cal L}_4)|=3$ and $|{\rm Con}_{\BI }({\cal L}_3\boxplus {\cal L}_3)|=|{\rm Con}_{\BI }({\cal L}_3\boxplus ({\cal L}_2\times {\cal L}_3))|=2$.

The pseudo--Kleene algebra in the rightmost diagram above, denoted by $B_6$, is called the {\em benzene ring}. Note that $B_6\cong {\cal L}_4\boxplus {\cal L}_4$, but $B_6\ncong _{\BI }{\cal L}_4\boxplus {\cal L}_4$, and that ${\rm Con}_{\BI }(B_6)=\{\Delta _{B_6},eq(\{0\},\{a,b^{\prime }\},\{b,a^{\prime }\},\{1\}),eq(\{0,a,b^{\prime }\},$\linebreak $\{b,a^{\prime },1\}), eq(\{0,a^{\prime },b\},\{b^{\prime },a,1\}),\nabla _{B_6}\}\cong {\cal L}_2\oplus {\cal L}_2^2$, so $|{\rm Con}_{\BI }(B_6)|=5$.\label{someex}\end{remark}

\begin{theorem}{\rm \cite{free,gcze,kumu}} Let $n\in \N ^*$ and $L$ be a lattice with $|L|=n$. Then:\begin{itemize}
\item $|{\rm Con}(L)|\leq 2^{n-1}$;
\item $|{\rm Con}(L)|=2^{n-1}$ iff $L\cong {\cal L}_n$ iff ${\rm Con}(L)\cong {\cal L}_2^{n-1}$;
\item $|{\rm Con}(L)|<2^{n-1}$ iff $n\geq 4$ and $|{\rm Con}(L)|\leq 2^{n-2}$;
\item $|{\rm Con}(L)|=2^{n-2}$ iff $n\geq 4$ and $L\cong {\cal L}_k\oplus {\cal L}_2^2\oplus {\cal L}_{n-k-2}$ for some $k\in [1,n-3]$ iff ${\rm Con}(L)\cong {\cal L}_2^{n-2}$;
\item $|{\rm Con}(L)|<2^{n-2}$ iff $n\geq 5$ and $|{\rm Con}(L)|\leq 5\cdot 2^{n-5}$;
\item $|{\rm Con}(L)|=5\cdot 2^{n-5}$ iff $n\geq 5$ and $L\cong {\cal L}_k\oplus N_5\oplus {\cal L}_{n-k-3}$ for some $k\in [1,n-4]$;
\item $|{\rm Con}(L)|<5\cdot 2^{n-5}$ iff $n\geq 6$ and $|{\rm Con}(L)|\leq 2^{n-3}$;
\item $|{\rm Con}(L)|=2^{n-3}$ iff $n\geq 6$ and $L\cong {\cal L}_k\oplus ({\cal L}_2\times {\cal L}_3)\oplus {\cal L}_{n-k-4}$ for some $k\in [1,n-5]$ or $n\geq 7$ and $L\cong {\cal L}_r\oplus {\cal L}_2^2\oplus {\cal L}_s\oplus {\cal L}_2^2\oplus {\cal L}_{n-r-s-4}$ for some $r,s\in \N ^*$ such that $r+s\leq n-5$;
\item $|{\rm Con}(L)|<2^{n-3}$ iff $n\geq 6$ and $|{\rm Con}(L)|\leq 7\cdot 2^{n-6}$;
\item $|{\rm Con}(L)|=7\cdot 2^{n-6}$ iff $n\geq 6$ and, for some $k\in [1,n-5]$, $L\cong {\cal L}_k\oplus ({\cal L}_3\boxplus {\cal L}_5)\oplus {\cal L}_{n-k-4}$ or $L\cong {\cal L}_k\oplus ({\cal L}_4\boxplus {\cal L}_4)\oplus {\cal L}_{n-k-4}$.\end{itemize}\label{maxcglat}\end{theorem}

By Theorem \ref{maxcglat}, for any $n\in \N ^*$ with $n\geq 6$, the first, second, third, fourth and fifth largest possible numbers of congruences of an $n$--element lattice are $2^{n-1}$, $2^{n-2}$, $5\cdot 2^{n-5}$, $2^{n-3}$ and $7\cdot 2^{n-6}$, respectively. In the following section we shall find the largest and number of congruences of an $n$--element i--lattice and that of an $n$--element BZ--lattice with the $0$ meet--irreducible, which is thus an antiortholattice by Lemma \ref{01irred}. To address this problem, let us first notice that, while, as we shall see, the finite i--lattices with the most full congruences also have the most lattice congruences, inequalities between the numbers of the lattice congruences of finite i--lattices with the same number of elements may be reversed between the numbers of their full congruences, as shown by the following example. 

\begin{example} Let us consider the following bi--lattices, whose numbers of congruences we will calculate using Lemmas \ref{thecg} and \ref{fortheex} and Proposition \ref{alsofortheex}.

$|{\cal L}_2\times {\cal L}_3|=|B_6|=6$. ${\rm Con}({\cal L}_2\times {\cal L}_3)\cong {\cal L}_2\times {\cal L}_2^2={\cal L}_2^3$. ${\rm Con}_{\BI }({\cal L}_2\times {\cal L}_3)\cong {\cal L}_2\times {\cal L}_2^{\lfloor 3/2\rfloor }={\cal L}_2^2$. ${\rm Con}(B_6)={\rm Con}({\cal L}_4\boxplus {\cal L}_4)={\rm Con}(({\cal L}_2\oplus {\cal L}_2\oplus {\cal L}_2)\boxplus ({\cal L}_2\oplus {\cal L}_2\oplus {\cal L}_2))\cong {\cal L}_2^2\oplus {\cal L}_2^2$. By Remark \ref{someex}, ${\rm Con}_{\BI }(B_6)\cong {\cal L}_2\oplus {\cal L}_2^2$. Hence $|{\rm Con}({\cal L}_2\times {\cal L}_3)|=8>7=|{\rm Con}(B_6)|$, while $|{\rm Con}_{\BI }({\cal L}_2\times {\cal L}_3)|=4<5=|{\rm Con}_{\BI }(B_6)|$.

Now let $L=(M_3\boxplus {\cal L}_4)\oplus {\cal L}_2^3\oplus (M_3\boxplus {\cal L}_4)$, $M={\cal L}_2^2\oplus {\cal L}_2\oplus {\cal L}_2^2$ and $H={\cal L}_4\boxplus {\cal L}_4\boxplus {\cal L}_4$, represented by the following Hasse diagrams:\vspace*{-10pt}

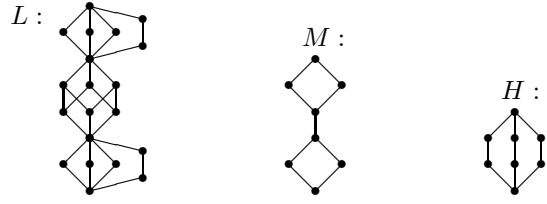
\begin{center}\begin{tabular}{ccc}
\begin{picture}(40,70)(0,0)
\put(-10,63){$L:$}
\put(20,0){\circle*{3}}
\put(20,10){\circle*{3}}
\put(10,10){\circle*{3}}
\put(40,5){\circle*{3}}
\put(40,15){\circle*{3}}
\put(30,10){\circle*{3}}
\put(20,20){\circle*{3}}
\put(20,0){\line(0,1){20}}
\put(20,0){\line(4,1){20}}
\put(20,20){\line(4,-1){20}}
\put(40,5){\line(0,1){10}}
\put(20,0){\line(1,1){10}}
\put(20,0){\line(-1,1){10}}
\put(20,20){\line(1,-1){10}}
\put(20,20){\line(-1,-1){10}}
\put(20,20){\circle*{3}}
\put(10,30){\circle*{3}}
\put(30,30){\circle*{3}}
\put(20,40){\circle*{3}}
\put(20,20){\line(0,1){10}}
\put(10,30){\line(0,1){10}}
\put(30,30){\line(0,1){10}}
\put(20,40){\line(0,1){10}}
\put(20,20){\line(1,1){10}}
\put(20,20){\line(-1,1){10}}
\put(20,40){\line(1,-1){10}}
\put(20,40){\line(-1,-1){10}}
\put(20,30){\circle*{3}}
\put(10,40){\circle*{3}}
\put(30,40){\circle*{3}}
\put(20,50){\circle*{3}}
\put(20,30){\line(1,1){10}}
\put(20,30){\line(-1,1){10}}
\put(20,50){\line(1,-1){10}}
\put(20,50){\line(-1,-1){10}}
\put(20,50){\circle*{3}}
\put(20,60){\circle*{3}}
\put(10,60){\circle*{3}}
\put(40,55){\circle*{3}}
\put(40,65){\circle*{3}}
\put(30,60){\circle*{3}}
\put(20,70){\circle*{3}}
\put(20,50){\line(1,1){10}}
\put(20,50){\line(-1,1){10}}
\put(20,70){\line(1,-1){10}}
\put(20,70){\line(-1,-1){10}}
\put(20,50){\line(0,1){20}}
\put(20,50){\line(4,1){20}}
\put(20,70){\line(4,-1){20}}
\put(40,55){\line(0,1){10}}
\end{picture} &\hspace*{30pt}
\begin{picture}(40,70)(0,0)
\put(15,55){$M:$}
\put(20,0){\circle*{3}}
\put(10,10){\circle*{3}}
\put(30,10){\circle*{3}}
\put(20,20){\circle*{3}}
\put(20,30){\circle*{3}}
\put(10,40){\circle*{3}}
\put(30,40){\circle*{3}}
\put(20,50){\circle*{3}}
\put(20,20){\line(0,1){10}}
\put(20,0){\line(1,1){10}}
\put(20,0){\line(-1,1){10}}
\put(20,20){\line(1,-1){10}}
\put(20,20){\line(-1,-1){10}}
\put(20,30){\line(1,1){10}}
\put(20,30){\line(-1,1){10}}
\put(20,50){\line(1,-1){10}}
\put(20,50){\line(-1,-1){10}}
\end{picture} &\hspace*{20pt}
\begin{picture}(40,70)(0,0)
\put(15,35){$H:$}
\put(20,0){\circle*{3}}
\put(20,10){\circle*{3}}
\put(20,20){\circle*{3}}
\put(20,30){\circle*{3}}
\put(10,10){\circle*{3}}
\put(30,10){\circle*{3}}
\put(10,20){\circle*{3}}
\put(30,20){\circle*{3}}
\put(20,0){\line(0,1){30}}
\put(20,0){\line(1,1){10}}
\put(20,0){\line(-1,1){10}}
\put(20,30){\line(1,-1){10}}
\put(20,30){\line(-1,-1){10}}
\put(10,10){\line(0,1){10}}
\put(30,10){\line(0,1){10}}
\end{picture}
\end{tabular}\end{center}\vspace*{-5pt}

$|M|=|H|=8$. ${\rm Con}(M)\cong {\cal L}_2^2\times {\cal L}_2\times {\cal L}_2^2={\cal L}_2^5$ and ${\rm Con}_{\BI }(M)\cong {\cal L}_2^2\times {\cal L}_2={\cal L}_2^3$. ${\rm Con}(H)={\rm Con}_{\BI }(H)\cong {\cal L}_2^3\oplus {\cal L}_2$. Hence $|{\rm Con}(M)|=32>9=|{\rm Con}(H)|$, while $|{\rm Con}_{\BI }(M)|=8<9=|{\rm Con}_{\BI }(H)|$.

$|{\cal L}_4\times {\cal L}_5|=|L|=20$. ${\rm Con}({\cal L}_4\times {\cal L}_5)\cong {\cal L}_2^3\times {\cal L}_2^4\cong {\cal L}_2^7$ and ${\rm Con}_{\BI }({\cal L}_4\times {\cal L}_5)\cong {\cal L}_2^{\lfloor 4/2\rfloor }\times {\cal L}_2^{\lfloor 5/2\rfloor }\cong {\cal L}_2^4$. ${\rm Con}(L)\cong ({\cal L}_2\oplus {\cal L}_2)\times {\cal L}_2^3\times ({\cal L}_2\oplus {\cal L}_2)\cong {\cal L}_3^2\times {\cal L}_2^3$ and ${\rm Con}_{\BI }(L)\cong ({\cal L}_2\oplus {\cal L}_2)\times {\cal L}_2^3\cong {\cal L}_3\times {\cal L}_2^3$. Hence $|{\rm Con}({\cal L}_4\times {\cal L}_5)|=128>72=|{\rm Con}(L)|$, while 
$|{\rm Con}_{\BI }({\cal L}_4\times {\cal L}_5)|=16<24=|{\rm Con}_{\BI }(L)|$.

Since every finite chain is a Kleene algebra, it follows that ${\cal L}_2\times {\cal L}_3$ and ${\cal L}_4\times {\cal L}_5$ are Kleene algebras. By Remark \ref{someex}, ${\cal L}_4\boxplus {\cal L}_4$ is not a pseudo--Kleene algebra, thus $H$ is not a pseudo--Kleene algebra. By Remark \ref{theconstr}, $L$ is a pseudo--Kleene algebra and $M$ is a Kleene algebra.

Let $n\in 2(\N \setminus \{0,1,2\})$ and $k\in 2\N ^*$ with $k\leq n-6$, arbitrary, and let us consider the $n$--element Kleene algebras $E_n={\cal L}_{n/2-2}\oplus ({\cal L}_2\times {\cal L}_3)\oplus {\cal L}_{n/2-2}$ and $E_{k,n}={\cal L}_{n/2-k/2-2}\oplus {\cal L}_2^2\oplus {\cal L}_k\oplus {\cal L}_2^2\oplus {\cal L}_{n/2-k/2-2}$, so that $E_{2,n}={\cal L}_{n/2-3}\oplus M\oplus {\cal L}_{n/2-3}$, the $n$--element pseudo--Kleene algebra $F_n={\cal L}_{n/2-2}\oplus B_6\oplus {\cal L}_{n/2-2}\cong {\cal L}_{n/2-2}\oplus ({\cal L}_4\boxplus {\cal L}_4)\oplus {\cal L}_{n/2-2}$ (recall that $\cong $ denotes the existence of a lattice isomorphism) and the $n$--element i--lattice $G_n={\cal L}_{n/2-3}\oplus H\oplus {\cal L}_{n/2-3}$. By Theorem \ref{maxcglat}, $E_n$ and $E_{k,n}$ have the fourth largest possible number of lattice congruences and $F_n$ has the fifth largest possible number of lattice congruences, while the number of lattice congruences of $G_n$ is not even in the top $5$ of the numbers of congruences of $n$--element lattices. But the above and Lemma \ref{thecg} show that $|{\rm Con}_{\BI }(E_{k,n})|=2^{n/2-k/2-3}\cdot 2^2\cdot 2^{k/2}=2^{n/2-1}=2^2\cdot 2^{n/2-3}=|{\rm Con}_{\BI }(E_n)|<|{\rm Con}_{\BI }(G_n)|=9\cdot 2^{n/2-4}<5\cdot 2^{n/2-3}=|{\rm Con}_{\BI }(F_n)|$.\end{example}

Let us see a property of atoms of distributive congruence lattices that we will use in the following section.

\begin{remark} If $L$ is a distributive lattice with a $0$ and $a\in {\rm At}(L)$, then $|L|\leq 2\cdot |[a)|$.

Indeed, let us define $h:L\rightarrow [a)$ by $h(x)=x\vee a$ for all $x\in L$, and consider $u,v,w\in L$ such that $h(u)=h(v)=h(w)$, that is $u\vee a=v\vee a=w\vee a$. Since $u\wedge a,v\wedge a,w\wedge a\in (a]=\{0,a\}$, it follows that $u\wedge a$, $v\wedge a$ and $w\wedge a$ are not pairwise distinct, so at least two of them coincide. Say, for instance, $u\wedge a=v\wedge a$. Then $u$ and $v$ are complements of $a$ in the bounded lattice $[u\wedge a,u\vee a]$, which is a sublattice of $L$ and thus a bounded distributive lattice, hence $u=v$. Therefore each element of $[a)$ has at most two distinct preimages through $h$, hence the statement above on cardinalities.\label{rtheat}\end{remark}

\begin{lemma} If $A$ is a congruence--distributive algebra from a variety $\V $ and $\alpha \in {\rm At}({\rm Con}_{\V }(A))$, then $|{\rm Con}_{\V }(A)|\leq 2\cdot |{\rm Con}_{\V }(A/\alpha )|$.\label{atoms}\end{lemma}

\begin{proof} Since the map $\theta \mapsto \theta /\alpha $ is a lattice isomorphism from $[\alpha )$ to ${\rm Con}_{\V }(A/\alpha )$, from Remark \ref{rtheat} we obtain that $|{\rm Con}_{\V }(A)|\leq 2\cdot |[\alpha )|=2\cdot |{\rm Con}_{\V }(A/\alpha )|$.\end{proof}

\begin{lemma}{\rm \cite{gcze}} For any non--trivial finite lattice $L$:\begin{itemize}
\item\label{lgcze1} ${\rm At}({\rm Con}(L))\neq \emptyset $ and $\{Cg_L(a,b)\ |\ [a,b]\in {\rm Nrw}(L)\}\subseteq {\rm At}({\rm Con}(L))\subseteq \{Cg(a,b)\ |\ a,b\in L,a\prec b\}$;
\item\label{lgcze3} for any $a,b\in L$ such that $a\prec b$: $[a,b]\in {\rm Nrw}(L)$ iff $L/Cg_L(a,b)=\{\{a,b\}\}\cup \{\{x\}\ |\ x\in L\setminus \{a,b\}\}$ iff $|L/Cg_L(a,b)|=|L|-1$;
\item\label{lgcze4} for any $a,b\in L$ such that $a\prec b$ and $|L/Cg_L(a,b)|=|L|-2$, we have one of the following situations:

$a$ is meet--reducible, case in which $a\prec c$ for some $c\in L\setminus \{b\}$ such that $b\prec b\vee c$, $c\prec b\vee c$ and $L/Cg_L(a,b)=\{\{a,b\},\{c,b\vee c\}\}\cup \{\{x\}\ |\ x\in L\setminus \{a,b,c,b\vee c\}\}$, so that $a/Cg_L(a,b)=b/Cg_L(a,b)\prec c/Cg_L(a,b)=(b\vee c) /Cg_L(a,b)$;

or $b$ is join--reducible, case in which, dually, $c\prec b$ for some $c\in L\setminus \{a\}$ such that $a\wedge c\prec a$, $a\wedge c\prec c$ and $L/Cg_L(a,b)=\{\{b\wedge c,c\},\{a,b\}\}\cup \{\{x\}\ |\ x\in L\setminus \{b\wedge c,c,a,b\}\}$, so that $(a\wedge c)/Cg_L(a,b)=c/Cg_L(a,b)\prec a/Cg_L(a,b)=b/Cg_L(a,b)$.\end{itemize}\label{lgcze}\end{lemma}

\begin{remark} For any i--lattice $L$, since $\cdot ^{\prime }$ is a dual lattice automorphism of $L$, it follows that the map $[a,b]\mapsto [b^{\prime },a^{\prime }]$ is a bijection from ${\rm Nrw}(L)$ to itself.\end{remark}

\begin{remark} If $L$ is a lattice, $x,y\in L$ and $\theta \in {\rm Con}(L)$ such that $x\prec y$ and $x/\theta \neq y/\theta $, then $x/\theta \prec y/\theta $.

If $L$ is a finite non--trivial i--lattice, then, for every $a,b\in L$ such that $a\prec b$, since $(a,b)\in Cg_L(a,b)\subseteq Cg_{L,\BI }(a,b)$, it follows that $|L/Cg_{L,\BI }(a,b)|\leq |L/Cg_L(a,b)|\leq |L|-1$, thus, according to Lemma \ref{lgcze}, if $[a,b]\notin {\rm Nrw}(L)$, then $|L/Cg_{L,\BI }(a,b)|\leq |L/Cg_L(a,b)|\leq |L|-2$.\label{onlgcze}\end{remark}

\begin{lemma} Let $L$ be an i--lattice. Then:\begin{itemize}
\item for all $\theta \in {\rm Con}(L)$, $Cg_{L,\I }(\theta )=\theta \vee \theta ^{\prime }$;
\item for all $U\subseteq L^2$, $Cg_{L,\I }(U)=Cg_L(U)\vee (Cg_L(U))^{\prime }=Cg_L(U)\vee Cg_L(U^{\prime })$;
\item for all $a,b\in L$, $Cg_{L,\I }(a,b)=Cg_L(a,b)\vee (Cg_L(a,b))^{\prime }=Cg_L(a,b)\vee Cg_L(a^{\prime },b^{\prime })$.\end{itemize}\label{cgbi}\end{lemma}

\begin{proof} Let $\theta \in {\rm Con}(L)$. Then $\theta \subseteq Cg_{L,\I }(\theta )$, hence $\theta ^{\prime }\subseteq Cg_{L,\I }(\theta )^{\prime }=Cg_{L,\I }(\theta )$ by Lemma \ref{thecg}, therefore $\theta \vee \theta ^{\prime }\subseteq Cg_{L,\I }(\theta )$. $\theta \subseteq \theta \vee \theta ^{\prime }\in {\rm Con}_{\I }(L)$ by Lemma \ref{thecg}, thus $Cg_{L,\I }(\theta )\subseteq \theta \vee \theta ^{\prime }$. Therefore $Cg_{L,\I }(\theta )=\theta \vee \theta ^{\prime }$.

Now let $U\subseteq L^2$. Then $U\subseteq Cg_L(U)\subseteq Cg_{L,\I }(U)$, thus $Cg_{L,\I }(U)\subseteq Cg_{L,\I }(Cg_L(U))\subseteq Cg_{L,\I }(Cg_{L,\I }(U))=Cg_{L,\I }(U)$, hence $Cg_{L,\I }(U)=Cg_{L,\I }(Cg_L(U))=Cg_L(U)\vee (Cg_L(U))^{\prime }=Cg_L(U)\vee Cg_L(U^{\prime })$, by the above.

In particular, for all $a,b\in L$, $Cg_{L,\I }(a,b)=Cg_L(a,b)\vee (Cg_L(a,b))^{\prime }=Cg_L(a,b)\vee Cg_L(a^{\prime },b^{\prime })$.\end{proof}

\begin{lemma} If $L$ is an i--lattice, then:\begin{enumerate}
\item\label{atlatbi6} ${\rm At}({\rm Con}(L))\cap {\rm Con}_{\I }(L)\subseteq {\rm At}({\rm Con}_{\I }(L))$;
\item\label{atlatbi2} ${\rm At}({\rm Con}_{\I }(L))=\{Cg_{L,\I }(\theta )\ |\ \theta \in {\rm At}({\rm Con}(L))\}=\{\theta \vee \theta ^{\prime }\ |\ \theta \in {\rm At}({\rm Con}(L))\}$ and, for any $\theta \in {\rm Con}(L)$, we have: $\theta \in {\rm At}({\rm Con}(L))$ iff $Cg_{L,\I }(\theta )=\theta \vee \theta ^{\prime }\in {\rm At}({\rm Con}_{\I }(L))$;
\item\label{atlatbi34} if $L$ is finite and non--trivial, then ${\rm At}({\rm Con}_{\BI }(L))\neq \emptyset $ and $\{Cg_{L,\BI }(a,b)\ |\ [a,b]\in {\rm Nrw}(L)\}\subseteq {\rm At}({\rm Con}_{\BI }(L))\subseteq \{Cg_{L,\BI }(a,b):a,b\in L,a\prec b\}$.\end{enumerate}\label{atlatbi}\end{lemma}

\begin{proof} (\ref{atlatbi6})  By Corollary \ref{cgred}, ${\rm Con}_{\I }(L)$ is a sublattice of ${\rm Con}(L)$, hence the inclusion in the enunciation.

\noindent (\ref{atlatbi2}) We will apply Lemma \ref{cgbi}.

Let $\theta \in {\rm At}({\rm Con}(L))$, so that $\theta ^{\prime }\in {\rm At}({\rm Con}(L)$, as well, and assume by absurdum that $\theta \vee \theta ^{\prime }\notin {\rm At}({\rm Con}_{\I }(L))$. If $\theta =\theta ^{\prime }$, then $\theta \in {\rm Con}_{\I }(L)$, so that $\theta \vee \theta ^{\prime }=\theta \in {\rm At}({\rm Con}_{\I }(L))$ by (\ref{atlatbi6}), and we have a contradiction. Hence $\theta \neq \theta ^{\prime }$, so that $\theta $ and $\theta ^{\prime }$ are incomparable. Since $\theta \in {\rm At}({\rm Con}(L)$, we have $\theta \neq \Delta _L$, thus $\theta \vee \theta ^{\prime }\neq \Delta _L$. Then there exists an $\alpha \in {\rm Con}_{\I}(L)$ such that $\Delta _L\subsetneq \alpha \subsetneq \theta \vee \theta ^{\prime }$. If $\theta \nsubseteq \alpha $, then $\theta ^{\prime }\nsubseteq \alpha $, because otherwise $\theta =\theta ^{\prime \prime }\leq \alpha ^{\prime }=\alpha $; since $\theta ,\theta ^{\prime }\in {\rm At}({\rm Con}(L)$, in this case $\alpha $ is incomparable to $\theta $ and to $\theta ^{\prime }$ and hence $\alpha \cap \theta =\alpha \cap \theta ^{\prime }=\Delta _L$, thus, since $\theta $ is incomparable to $\theta ^{\prime }$ and hence $\theta \cap \theta ^{\prime }=\Delta _L$, it follows that $M_3\cong \{\Delta _L,\alpha ,\theta ,\theta ^{\prime },\theta \vee \theta ^{\prime }\}$ is a sublattice of ${\rm Con}(L)$, contradicting the distributivity of ${\rm Con}(L)$. Therefore $\theta \leq \alpha $; but then $\theta ^{\prime }\leq \alpha ^{\prime }=\alpha $, so that $\theta \vee \theta ^{\prime }\leq \alpha $, which contradicts $\alpha \subsetneq \theta \vee \theta ^{\prime }$. Therefore $\theta \vee \theta ^{\prime }\in {\rm At}({\rm Con}_{\I }(L))$, so $\{\theta \vee \theta ^{\prime }:\theta \in {\rm At}({\rm Con}(L))\}\subseteq {\rm At}({\rm Con}_{\I }(L))$.

Now let $\alpha \in {\rm At}({\rm Con}_{\I }(L))\subseteq {\rm Con}_{\I }(L)\setminus \{\Delta _L\}$. If there exists no $\theta \in {\rm Con}(L)\setminus \{\Delta _L\}$ such that $\theta \subseteq \alpha $, then $\alpha \in {\rm At}({\rm Con}(L))$ and thus $\alpha =\alpha \vee \alpha ^{\prime }\in \{\theta \vee \theta ^{\prime }:\theta \in {\rm At}({\rm Con}(L))\}$. Otherwise, we have $\Delta _L\subsetneq \theta \subsetneq \alpha $ for some $\theta \in {\rm Con}(L)$, so that $\Delta _L\subsetneq \theta ^{\prime }\subsetneq \alpha ^{\prime }=\alpha $ and thus $\Delta _L\subsetneq \theta \vee \theta ^{\prime }\subseteq \alpha $. But $\theta \vee \theta ^{\prime }\in {\rm Con}_{\I }(L)$ and $\alpha \in {\rm At}({\rm Con}_{\I }(L))$, thus $\theta \vee \theta ^{\prime }=\alpha $. If $\theta \notin {\rm At}({\rm Con}(L))$, then $\Delta _L\subsetneq \phi \subsetneq \theta $ for some $\phi \in {\rm Con}(L)$, so that $\Delta _L\subsetneq \phi ^{\prime }\subsetneq \theta ^{\prime }$ and thus $\Delta _L\subsetneq \phi \vee \phi ^{\prime }\subseteq \theta \vee \theta ^{\prime }=\alpha $, so that $\phi \vee \phi ^{\prime }=\alpha $ since $\alpha \in {\rm At}({\rm Con}_{\I }(L))$. If $\theta =\theta ^{\prime }$, then $\alpha =\theta \vee \theta ^{\prime }=\theta \subsetneq \alpha $ and we have a contradiction. Hence $\theta \neq \theta ^{\prime }$, so that $\theta $ and $\theta ^{\prime }$ are incomparable, thus $\theta \cap \theta ^{\prime }\subsetneq \theta \subsetneq \theta \vee \theta ^{\prime }=\alpha $. But $\theta \cap \theta ^{\prime }\in {\rm Con}_{\I }(L)$ by Lemma \ref{cginvlat}, thus $\theta \cap \theta ^{\prime }=\Delta _L$ since $\alpha \in {\rm At}({\rm Con}_{\I }(L))$. Since $\phi \cap \phi ^{\prime }\subseteq \theta \cap \theta ^{\prime }$, it follows that $\phi \cap \phi ^{\prime }=\Delta _L$. We also have $\alpha =\phi \vee \phi ^{\prime }\subseteq \theta \vee \phi ^{\prime }\subseteq \theta \cap \theta ^{\prime }=\alpha $, thus $\theta \vee \phi ^{\prime }=\alpha $. Therefore $N_5\cong \{\Delta _L,\phi ,\phi ^{\prime },\theta ,\alpha \}$ is a sublattice of ${\rm Con}(L)$, contradicting the distributivity of ${\rm Con}(L)$. Therefore $\theta \in {\rm At}({\rm Con}(L))$, hence the converse of the inclusion above holds, as well. The above also show that, if $\theta \in {\rm Con}(L)$ is such that $\theta \vee \theta ^{\prime }\in {\rm At}({\rm Con}_{\I }(L))$, then $\theta \in {\rm At}({\rm Con}(L))$, since we can not have $\theta =\Delta _L$.

\noindent (\ref{atlatbi34}) By (\ref{atlatbi2}) and Lemma \ref{lgcze}.\end{proof}

\begin{proposition} If $L$ is an i--lattice, then:\begin{itemize}
\item $L$ is subdirectly irreducible iff $|{\rm At}({\rm Con}_{\I }(L))|=1$ iff ${\rm At}({\rm Con}(L))=\{\alpha ,\alpha ^{\prime }\}$ for some $\alpha \in {\rm Con}(L)$;
\item if $L$ is finite and non--trivial, then: $L$ is subdirectly irreducible iff ${\rm At}({\rm Con}_{\BI }(L))=\{Cg_{L,\BI }(a,b)\}$ for some $a,b\in L$ such that $a\prec b$ iff ${\rm At}({\rm Con}(L))=\{Cg_L(a,b),Cg_L(a^{\prime },b^{\prime })\}$ for some $a,b\in L$ such that $a\prec b$ iff ${\rm Nrw}(L)\subseteq \{[a,b],[b^{\prime },a^{\prime }]\}$ for some $a,b\in L$.\end{itemize}\label{bisi}\end{proposition}

\begin{proof} We apply Lemma \ref{atlatbi}, (\ref{atlatbi2}). Recall that $L$ is subdirectly irreducible iff ${\rm Con}_{\BI }(L)$ has a single atom iff, for all $\alpha ,\beta \in {\rm At}({\rm Con}(L))$, $\alpha \vee \alpha ^{\prime }=\beta \vee \beta ^{\prime }$; now let us prove that the latter property is equivalent to ${\rm At}({\rm Con}(L))=\{\theta ,\theta ^{\prime }\}$ for some $\theta \in {\rm At}({\rm Con}(L))$. The converse implication follows from the fact that $\theta \vee \theta ^{\prime }=\theta ^{\prime }\vee \theta ^{\prime \prime }$ for all $\theta \in {\rm Con}(L)$. For the direct implication, let us consider $\alpha ,\beta \in {\rm At}({\rm Con}(L))$ such that $\alpha \vee \alpha ^{\prime }=\beta \vee \beta ^{\prime }$. If $\alpha \in \{\beta ,\beta ^{\prime }\}$, then $\alpha ^{\prime }\in \{\beta ,\beta ^{\prime }\}$, as well, so $\{\alpha ,\alpha ^{\prime }\}=\{\beta ,\beta ^{\prime }\}$. Now assume that $\alpha \notin \{\beta ,\beta ^{\prime }\}$, case in which we can have neither $\alpha =\alpha ^{\prime }$, nor $\beta =\beta ^{\prime }$, because then either $\alpha =\alpha \vee \alpha ^{\prime }=\beta \vee \beta ^{\prime }=\beta $ or $\alpha \subsetneq \alpha \vee \alpha ^{\prime }=\beta \vee \beta ^{\prime }=\beta $ or $\beta \subsetneq \beta \vee \beta ^{\prime }=\alpha \vee \alpha ^{\prime }=\alpha $, contradicting the facts that $\alpha \notin \{\beta ,\beta ^{\prime }\}$, $\beta \in {\rm At}({\rm Con}(L))$, respectively $\alpha \in {\rm At}({\rm Con}(L))$. Then we can not have $\alpha \vee \beta =\beta \vee \beta ^{\prime }$, because this would imply $\alpha \vee \beta ^{\prime }=(\alpha \vee \beta ^{\prime })\cap (\alpha \vee \beta \vee \beta ^{\prime })=(\alpha \vee \beta ^{\prime })\cap (\beta \vee \beta ^{\prime })=(\alpha \cap \beta )\vee \beta ^{\prime }=\Delta _L\vee \beta ^{\prime }=\beta ^{\prime }$, which contradicts the fact that $\alpha \nsubseteq \beta ^{\prime }$ since $\alpha $ and $\beta ^{\prime }$ are different atoms of ${\rm Con}(L)$. But $\alpha \subset \alpha \vee \alpha ^{\prime }=\beta \vee \beta ^{\prime }\supset \beta $, hence $\alpha \vee \beta \subseteq \beta \vee \beta ^{\prime }$, thus $\alpha \vee \beta \subsetneq \beta \vee \beta ^{\prime }$ since $\alpha \vee \beta \neq \beta \vee \beta ^{\prime }$, and $\alpha \vee \beta \vee \beta ^{\prime }=\beta \vee \beta ^{\prime }=\alpha \vee \alpha ^{\prime }$. Also, $(\alpha \vee \beta )\cap \beta ^{\prime }=(\alpha \cap \beta ^{\prime })\vee (\beta \cap \beta ^{\prime })=\Delta _L\vee \Delta _L=\Delta _L$ and, of course, $\beta \subsetneq \alpha \vee \beta $. Therefore $N_5\cong \{\Delta _L,\beta ,\beta ^{\prime },\alpha \vee \beta ,\beta \vee \beta ^{\prime }\}$ is a sublattice of ${\rm Con}(L)$, contradicting the distributivity of ${\rm Con}(L)$. Hence, by Lemma \ref{lgcze}, if $L$ is finite and non--trivial, then: $L$ is subdirectly irreducible iff ${\rm At}({\rm Con}(L))=\{\theta ,\theta ^{\prime }\}$ for some $\theta \in {\rm At}({\rm Con}(L))$ iff ${\rm At}({\rm Con}(L))=\{Cg_L(a,b),Cg_L(a^{\prime },b^{\prime })\}$ for some $a,b\in L$ with $a\prec b$, iff ${\rm At}({\rm Con}_{\BI }(L))=\{Cg_{L,\BI }(a,b)\}$ for some $a,b\in L$ with $a\prec b$, iff, for some $a,b\in L$, ${\rm Nrw}(L)\subseteq \{[a,b],[b^{\prime },a^{\prime }]\}$.\end{proof}

\begin{remark} By Lemma \ref{01irred}, if $L$ is a finite BZ--lattice with the $0$ meet--irreducible, then $L$ is an antiortholattice with the $0$ strictly meet--irreducible, thus also the $1$ strictly join--irreducible, so that $L={\cal L}_2\oplus K\oplus {\cal L}_2$ for some pseudo--Kleene algebra $K$. In this case, obviously, $0^{+}$ is join--irreducible and $1^{-}$ is meet--irreducible in $L$, thus $[0,0^{+}],[1^{-},1]\in {\rm Nrw}(L)$.\label{lsiiffksi}\end{remark}

\begin{corollary} If $L$ is a finite BZ--lattice with $0$ meet--irreducible, then: $L$ is subdirectly irreducible iff ${\rm At}({\rm Con}(L))=\{Cg_L(0,0^{+}),Cg_L(a,b),Cg_(a^{\prime },b^{\prime }),Cg_L(1^{-},1)\}$ for some $a,b\in L\setminus \{0,1\}$ such that $a\prec b$ iff ${\rm At}({\rm Con}_{\BI }(L))=\{Cg_{L,\BI }(0,0^{+}),Cg_{L,\BI }(a,b)\}$ for some $a,b\in L\setminus \{0,1\}$ such that $a\prec b$ iff ${\rm Nrw}(L)\subseteq \{[0,0^{+}],[a,b],$\linebreak $[b^{\prime },a^{\prime }],[1^{-},1]\}$ for some $a,b\in L$.\end{corollary}

\begin{proof} By Remarks \ref{lksi} and \ref{lsiiffksi} and Proposition \ref{bisi}.\end{proof}

\section{The Largest Numbers of Congruences of Finite i--lattices}

\begin{remark} Let $L$ be a finite i--lattice and $a,b\in L$. Then, by Lemma \ref{cgbi} and Theorem \ref{cgeq}, $Cg_{L,\BI }(a,b)=Cg_L(a,b)\vee Cg_L(a^{\prime },b^{\prime })=eq(L/Cg_L(a,b))\vee eq(L/Cg_L(a^{\prime },b^{\prime }))=eq(L/Cg_L(a,b)\vee L/Cg_L(a^{\prime },b^{\prime }))$, where the latter join is the join in the lattice ${\rm Part}(L)$.\vspace*{-12pt}

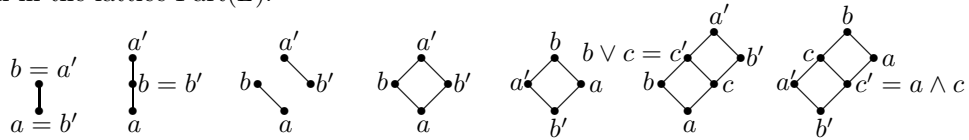
\begin{center}\begin{tabular}{ccccccc}
\hspace*{-45pt}
\begin{picture}(40,40)(0,0)
\put(20,0){\circle*{3}}
\put(20,10){\circle*{3}}
\put(20,0){\line(0,1){10}}
\put(9,-8){$a=b^{\prime }$}
\put(9,13){$b=a^{\prime }$}
\end{picture} &\hspace*{-20pt}
\begin{picture}(40,40)(0,0)
\put(20,0){\circle*{3}}
\put(20,10){\circle*{3}}
\put(20,20){\circle*{3}}
\put(20,0){\line(0,1){20}}
\put(18,-8){$a$}
\put(22,7){$b=b^{\prime }$}
\put(18,23){$a^{\prime }$}
\end{picture} &\hspace*{2pt}
\begin{picture}(40,40)(0,0)
\put(20,0){\circle*{3}}
\put(10,10){\circle*{3}}
\put(30,10){\circle*{3}}
\put(20,20){\circle*{3}}
\put(20,0){\line(-1,1){10}}
\put(20,20){\line(1,-1){10}}
\put(18,-8){$a$}
\put(3,8){$b$}
\put(32,7){$b^{\prime }$}
\put(18,23){$a^{\prime }$}
\end{picture} &
\begin{picture}(40,40)(0,0)
\put(20,0){\circle*{3}}
\put(10,10){\circle*{3}}
\put(30,10){\circle*{3}}
\put(20,20){\circle*{3}}
\put(20,0){\line(1,1){10}}
\put(20,0){\line(-1,1){10}}
\put(20,20){\line(1,-1){10}}
\put(20,20){\line(-1,-1){10}}
\put(18,-8){$a$}
\put(3,8){$b$}
\put(32,7){$b^{\prime }$}
\put(18,23){$a^{\prime }$}
\end{picture} &\hspace*{-5pt}
\begin{picture}(40,40)(0,0)
\put(20,0){\circle*{3}}
\put(10,10){\circle*{3}}
\put(30,10){\circle*{3}}
\put(20,20){\circle*{3}}
\put(20,0){\line(1,1){10}}
\put(20,0){\line(-1,1){10}}
\put(20,20){\line(1,-1){10}}
\put(20,20){\line(-1,-1){10}}
\put(18,-10){$b^{\prime }$}
\put(3,8){$a^{\prime }$}
\put(33,7){$a$}
\put(18,23){$b$}
\end{picture} &\hspace*{-5pt}
\begin{picture}(40,40)(0,0)
\put(20,0){\circle*{3}}
\put(10,10){\circle*{3}}
\put(30,10){\circle*{3}}
\put(20,20){\circle*{3}}
\put(40,20){\circle*{3}}
\put(30,30){\circle*{3}}
\put(20,0){\line(-1,1){10}}
\put(20,20){\line(1,-1){10}}
\put(10,10){\line(1,1){20}}
\put(20,0){\line(1,1){20}}
\put(30,30){\line(1,-1){10}}
\put(18,-8){$a$}
\put(33,7){$c$}
\put(3,8){$b$}
\put(-20,19){$b\vee c=c^{\prime }$}
\put(42,17){$b^{\prime }$}
\put(28,33){$a^{\prime }$}
\end{picture} &\hspace*{-5pt}
\begin{picture}(40,40)(0,0)
\put(20,0){\circle*{3}}
\put(10,10){\circle*{3}}
\put(30,10){\circle*{3}}
\put(20,20){\circle*{3}}
\put(40,20){\circle*{3}}
\put(30,30){\circle*{3}}
\put(20,0){\line(-1,1){10}}
\put(20,20){\line(1,-1){10}}
\put(10,10){\line(1,1){20}}
\put(20,0){\line(1,1){20}}
\put(30,30){\line(1,-1){10}}
\put(18,-10){$b^{\prime }$}
\put(33,7){$c^{\prime }=a\wedge c$}
\put(3,8){$a^{\prime }$}
\put(13,19){$c$}
\put(43,17){$a$}
\put(28,33){$b$}
\end{picture}
\end{tabular}\end{center}\vspace*{2pt}

Assume that $[a,b]\in {\rm Nrw}(L)$, so that $[b^{\prime },a^{\prime }]\in {\rm Nrw}(L)$, as well, thus, according to Lemma \ref{lgcze}, $Cg_L(a,b)=eq(\{\{a,b\}\}\cup \{\{x\}\ |\ x\in L\setminus \{a,b\}\})$ and $Cg_L(a^{\prime },b^{\prime })=eq(\{\{a^{\prime },b^{\prime }\}\}\cup \{\{y\}\ |\ y\in L\setminus \{a^{\prime },b^{\prime }\}\})$, hence $Cg_{L,\BI }(a,b)=Cg_L(a,b)=eq((\{\{a,b\}\}\cup \{\{x\}\ |\ x\in L\setminus \{a,b\}\})\vee (\{\{a^{\prime },b^{\prime }\}\}\cup \{\{y\}\ |\ y\in L\setminus \{a^{\prime },b^{\prime }\}\}))$, and we are in one of the following situations:\begin{itemize}
\item $a=b^{\prime }$, so $b=a^{\prime }$, as in the first diagram above, hence $Cg_L(a,b)=Cg_L(a^{\prime },b^{\prime })$, therefore $Cg_{L,\BI }(a,b)=eq((\{\{a,b\}\}\cup \{\{x\}\ |\ x\in L\setminus \{a,b\}\})$, so that $|L/Cg_{L,\BI }(a,b)|=|L|-1$;
\item $a\prec b=b^{\prime }\prec a^{\prime }$, as in the second diagram above, hence $Cg_{L,\BI }(a,b)=eq((\{\{a,b,a^{\prime }\}\}\cup \{\{x\}\ |\ x\in L\setminus \{a,b,a^{\prime }\}\})$, so that $|L/Cg_{L,\BI }(a,b)|=|L|-2$;
\item $\{a,b\}\cap \{a^{\prime },b^{\prime }\}=\emptyset $, as in the third diagram above, hence $Cg_{L,\BI }(a,b)=eq((\{\{a,b\},\{a^{\prime },b^{\prime }\}\}\cup \{\{x\}\ |\ x\in L\setminus \{a,b,a^{\prime },b^{\prime }\}\})$, so that $|L/Cg_{L,\BI }(a,b)|=|L|-2$.\end{itemize}

Now assume that $[a,b]\notin {\rm Nrw}(L)$, but $a\prec b$ and $|L/Cg_{L,\BI }(a,b)|=|L|-2$. Assume, for instance, that $a$ is meet--reducible; the case when $b$ is join--reducible follows by duality. Since $\cdot ^{\prime }:L\rightarrow L$ is a dual lattice isomorphism, it follows that $b^{\prime }\prec a^{\prime }$ and $a^{\prime }$ is join--reducible. 

Then, since $L$ is finite, there exists a $c\in L\setminus \{b\}$ with $a\prec b$, so that $b\nleq c$ and $c\nleq b$, thus $b<b\vee c$ and $c<b\vee c$. But $a/Cg_L(a,b)=b/Cg_L(a,b)$ and thus $c/Cg_L(a,b)=(a\vee c)/Cg_L(a,b)=(b\vee c)/Cg_L(a,b)$, hence $|L|-2\geq |L/Cg_L(a,b)|\geq |L/Cg_{L,\BI }(a,b)|=|L|-2$, therefore $|L/Cg_L(a,b)|=|L|-2=|L/Cg_{L,\BI }(a,b)|$ and thus, since $Cg_L (a,b)\subseteq Cg_{L,\BI }(a,b)$, we have $Cg_L (a,b)=Cg_{L,\BI }(a,b)$ and hence $Cg_L(a,b)\vee Cg_L(a^{\prime },b^{\prime })=Cg_{L,\BI }(a,b)=\! Cg_L(a,b)=eq(\{\{a,b\},\{c,b\vee c\}\}\cup \{\{x\}\ |\ x\in L\setminus \{a,b,c,b\vee c\}\})$ by Lemma \ref{lgcze}, which also ensures us that $b\prec b\vee c$, $c\prec b\vee c$ and $a/Cg_L(a,b)=b/Cg_L(a,b)\prec c/Cg_L(a,b)=(b\vee c)/Cg_L(a,b)$.

By the above, $Cg_L(a,b)\vee Cg_L(a^{\prime },b^{\prime })=Cg_L(a,b)$, hence $Cg_L(a,b)^{\prime }=Cg_L(a^{\prime },b^{\prime })\subseteq Cg_L(a,b)$, thus $Cg_L(a,b)=Cg_L(a,b)^{\prime \prime }\subseteq Cg_L(a,b)^{\prime }$, so $Cg_L(a,b)^{\prime }=Cg_L(a,b)=eq(\{\{a,b\},\{c,b\vee c\}\}\cup \{\{x\}\ |\ x\in L\setminus \{a,b,c,b\vee c\}\})$. But $a^{\prime }/Cg_L(a^{\prime },b^{\prime })=b^{\prime }/Cg_L(a^{\prime },b^{\prime })$, thus $a^{\prime }/Cg_L(a^{\prime },b^{\prime })$ is not a singleton, which means that $a^{\prime }/Cg_L(a^{\prime },b^{\prime })=\{a^{\prime },b^{\prime }\}=\{c,b\vee c\}$, thus $a^{\prime }=b\vee c$ and $b^{\prime }=c$ since $b ^{\prime }<a^{\prime }$. Therefore, in this case, $Cg_{L,\BI }(a,b)=Cg_L(a,b)= Cg_L(a^{\prime },b^{\prime })=eq((\{\{a,b\},\{a^{\prime },b^{\prime }\}\}\cup \{\{x\}\ |\ x\in L\setminus \{a,b,a^{\prime },b^{\prime }\}\})$ and hence $\{a,b,b^{\prime },a^{\prime }\}\cong {\cal L}_2^2$ is a convex sublattice of $L$, as in the fourth diagram above. The case dual to this one is represented in the fifth diagram above.

Finally, assume that $a\prec b$ and $|L/Cg_{L,\BI }(a,b)|=|L|-3$, which, by the above, can not hold if $[a,b]$ is a narrows. Hence $[a,b]$ is not a narrows, thus, as above, we can assume that $a$ is meet--reducible, so that $a\prec c$ for some $c\in L\setminus \{b\}$, so $b<b\vee c$ and $c<b\vee c$, thus $(b\vee c)^{\prime }<b^{\prime }$ and $(b\vee c)^{\prime }<c^{\prime }$. Since $a/Cg_{L,\BI }(a,b)=b/Cg_{L,\BI }(a,b)$, $a^{\prime }/Cg_{L,\BI }(a,b)=b^{\prime }/Cg_{L,\BI }(a,b)$, $c/Cg_{L,\BI }(a,b)=(b\vee c)/Cg_{L,\BI }(a,b)$ and $c^{\prime }/Cg_{L,\BI }(a,b)=(b\vee c)^{\prime }/Cg_{L,\BI }(a,b)$, $Cg_{L,\BI }(a,b)=eq(L/Cg_L(a,b)\vee L/Cg_L(a^{\prime },b^{\prime }))$ and $|L/Cg_{L,\BI }(a,b)|=|L|-3$, it follows that $a/Cg_{L,\BI }(a,b)=\{a,b\}\neq c/Cg_{L,\BI }(a,b) =\{c,b\vee c\}=\{c^{\prime },(b\vee c)^{\prime }\}\neq a^{\prime }/Cg_{L,\BI }(a,b)=\{a^{\prime },b^{\prime }\}$, hence $c^{\prime }=b\vee c$ and $|a/Cg_{L,\BI }(a,b)|=|c/Cg_{L,\BI }(a,b)|=|a^{\prime }/Cg_{L,\BI }(a,b)|=2$, so that $c\prec b\vee c$ and thus $a/Cg_{L,\BI }(a,b)=b/Cg_{L,\BI }(a,b)\prec c/Cg_{L,\BI }(a,b)=(b\vee c)/Cg_{L,\BI }(a,b)\prec a^{\prime }/Cg_{L,\BI }(a,b)=b^{\prime }/Cg_{L,\BI }(a,b)$, and thus $Cg_{L,\BI }(a,b)=Cg_L(a,b)=Cg_L(a^{\prime },b^{\prime })=eq(\{\{a,b\},\{c,c^{\prime }\},\{b^{\prime },a^{\prime }\}\}\cup \{\{x\}\ |\ x\in L\setminus \{a,b,c,a^{\prime },b^{\prime },c^{\prime }\}\})$, hence $\{a,b,c,a^{\prime },b^{\prime },c^{\prime }\}\cong {\cal L}_2\times {\cal L}_3$ is a convex sublattice of $L$, as in the sixth diagram above. The case dual to this one is represented in the rightmost diagram above.\label{narrows}\end{remark}

\begin{theorem} Let $n\in \N ^*$ and $L$ be an i--lattice with $|L|=n$. Then:\begin{itemize}

\item $|{\rm Con}_{\BI }(L)|\leq 2^{\lfloor n/2\rfloor }$;

\item $|{\rm Con}_{\BI }(L)|=2^{\lfloor n/2\rfloor }$ iff ${\rm Con}_{\BI }(L)\cong {\cal L}_2^{\lfloor n/2\rfloor }$ iff either $L\cong {\cal L}_n$ or $n\in 2(\N \setminus \{0,1\})$ and $L\cong _{\BI }{\cal L}_{n/2-1}\oplus {\cal L}_2^2\oplus {\cal L}_{n/2-1}$ iff either $|{\rm Con}(L)|=2^{n-1}$ or $n\geq 2$, $L\ncong _{\BI }{\cal L}_{\lfloor n/2\rfloor -1}\oplus ({\cal L}_3\boxplus {\cal L}_3)\oplus {\cal L}_{\lfloor n/2\rfloor -1}$ and $|{\rm Con}(L)|=2^{n-2}$ iff either ${\rm Con}(L)\cong {\cal L}_2^{n-1}$ or $n\geq 2$, $L\ncong _{\BI }{\cal L}_{\lfloor n/2\rfloor -1}\oplus ({\cal L}_3\boxplus {\cal L}_3)\oplus {\cal L}_{\lfloor n/2\rfloor -1}$ and ${\rm Con}(L)\cong {\cal L}_2^{n-2}$;

\item if $L$ is a pseudo--Kleene algebra, then: $|{\rm Con}_{\BI }(L)|=2^{\lfloor n/2\rfloor }$ iff ${\rm Con}_{\BI }(L)\cong {\cal L}_2^{\lfloor n/2\rfloor }$ iff either $L\cong {\cal L}_n$ or $n\in 2(\N \setminus \{0,1\})$ and $L\cong {\cal L}_{n/2-1}\oplus {\cal L}_2^2\oplus {\cal L}_{n/2-1}$ iff $|{\rm Con}(L)|\in \{2^{n-1},2^{n-2}\}$ iff either ${\rm Con}(L)\cong {\cal L}_2^{n-1}$ or $n\geq 2$ and ${\rm Con}(L)\cong {\cal L}_2^{n-2}$;

\item $|{\rm Con}_{\BI }(L)|<2^{\lfloor n/2\rfloor }$ iff either $n\in 2(\N \setminus \{0,1\})$ and $L\cong _{\BI }{\cal L}_{n/2-1}\oplus ({\cal L}_3\boxplus {\cal L}_3)\oplus {\cal L}_{n/2-1}$ or $|{\rm Con}(L)|<2^{n-2}$;

\item if $L$ is a pseudo--Kleene algebra, then: $|{\rm Con}_{\BI }(L)|<2^{\lfloor n/2\rfloor }$ iff $|{\rm Con}(L)|<2^{n-2}$ iff $|{\rm Con}(L)|\leq 2^{n-3}$;

\item if ${\rm Con}_{\BI }(L)$ is a Boolean algebra, in particular if ${\rm Con}(L)$ is a Boolean algebra, in particular if $L$ is a modular i--lattice, then: $|{\rm Con}_{\BI }(L)|<2^{\lfloor n/2\rfloor }$ iff $|{\rm Con}_{\BI }(L)|\leq 2^{\lfloor n/2\rfloor -1}$ iff either $n\in 2(\N \setminus \{0,1\})$ and $L\cong _{\BI }{\cal L}_{n/2-1}\oplus ({\cal L}_3\boxplus {\cal L}_3)\oplus {\cal L}_{n/2-1}$ or $|{\rm Con}(L)|<2^{n-2}$ iff either $n\in 2(\N \setminus \{0,1\})$ and $L\cong _{\BI }{\cal L}_{n/2-1}\oplus ({\cal L}_3\boxplus {\cal L}_3)\oplus {\cal L}_{n/2-1}$ or $|{\rm Con}(L)|\leq 2^{n-3}$;

\item if $L$ is a pseudo--Kleene algebra such that ${\rm Con}_{\BI }(L)$ is a Boolean algebra, in particular if $L$ is a pseudo--Kleene algebra such that ${\rm Con}(L)$ is a Boolean algebra, in particular if $L$ is a modular pseudo--Kleene algebra, then: $|{\rm Con}_{\BI }(L)|<2^{\lfloor n/2\rfloor }$ iff $|{\rm Con}_{\BI }(L)|\leq 2^{\lfloor n/2\rfloor -1}$ iff $|{\rm Con}(L)|<2^{n-2}$ iff $|{\rm Con}(L)|\leq 2^{n-3}$.\end{itemize}\label{maxcgkl}\end{theorem}

\begin{proof} For any $n\in \N ^*$, ${\rm Con}_{\BI }({\cal L}_n)\cong {\cal L}_2^{\lfloor n/2\rfloor }$; for any $n\in 2(\N \setminus \{0,1\})$, ${\rm Con}_{\BI }({\cal L}_{n/2-1}\oplus {\cal L}_2^2\oplus {\cal L}_{n/2-1})\cong {\cal L}_2^{\lfloor n/2\rfloor }$. For any $n\in 2(\N \setminus \{0,1\})$, the only ways in which the lattice ${\cal L}_{n/2-1}\oplus {\cal L}_2^2\oplus {\cal L}_{n/2-1}$ can be organized as an i--lattice is as ${\cal L}_{n/2-1}\oplus {\cal L}_2^2\oplus {\cal L}_{n/2-1}$, with ${\cal L}_2^2$ being the four--element Boolean algebra, or as ${\cal L}_{n/2-1}\oplus ({\cal L}_3\boxplus {\cal L}_3)\oplus {\cal L}_{n/2-1}$, the latter of which is not a pseudo--Kleene algebra.

${\rm Con}_{\BI }({\cal L}_1)={\rm Con}({\cal L}_1)\cong {\cal L}_1$, ${\rm Con}_{\BI }({\cal L}_2)={\rm Con}({\cal L}_2)\cong {\cal L}_2$, ${\rm Con}({\cal L}_3)\cong {\cal L}_2^2$ and ${\rm Con}_{\BI }({\cal L}_3)\cong {\cal L}_2$, ${\rm Con}({\cal L}_4)\cong {\cal L}_2^3$ and ${\rm Con}_{\BI }({\cal L}_4)\cong {\cal L}_2^2$. Since ${\cal L}_2^2$ is a Boolean algebra, we have ${\rm Con}_{\BI }({\cal L}_2^2)={\rm Con}({\cal L}_2^2)\cong {\cal L}_2^2$. Therefore the statements in the enunciation hold for every $n\in [1,4]$. Now assume that $n\geq 5$ and all the statements in the enunciation hold for all i--lattices having at most $n-1$ elements. Then $L$ is non--trivial, so that there exist $a,b\in L$ with $a\prec b$ and $Cg_L(a,b)\in {\rm At}({\rm Con}(L))$, according to Lemma \ref{lgcze}. Let us denote by $\theta =Cg_{L,\BI }(a,b)=Cg_L(a,b)\vee Cg_L(a^{\prime },b^{\prime })\in {\rm At}({\rm Con}_{\BI }(L))$ by Lemma \ref{atlatbi}, (\ref{atlatbi2}). In particular, $\theta \in {\rm Con}_{\BI }(L)$, so that $L/\theta $ is a bi--lattice and it is a pseudo--Kleene algebra if $L$ is a pseudo--Kleene algebra. Of course, $|L/\theta |\leq |L|-1=n-1$ since $a/\theta =b/\theta $, therefore $L/\theta $ fulfills the statements in the enunciation. We shall use Theorem \ref{maxcglat}, along with Lemma \ref{atoms}, which ensures us that $|{\rm Con}(L)|\leq 2\cdot |{\rm Con}(L/\theta )|$ and $|{\rm Con}_{\BI }(L)|\leq 2\cdot |{\rm Con}_{\BI }(L/\theta )|$; to determine the structure of $L$ based on that of $L/\theta $, we will use Remark \ref{narrows} and the fact that $L$ is self--dual; to calculate the numbers of congruences, we shall use Lemma \ref{thecg} and Remark \ref{someex}.

Assume by absurdum that $|{\rm Con}_{\BI }(L)|>2^{\lfloor n/2\rfloor }$, so that $|{\rm Con}_{\BI }(L/\theta )|>2^{\lfloor n/2\rfloor }/2=2^{\lfloor n/2\rfloor -1}=2^{\lfloor (n-2)/2\rfloor }$, thus, by the induction hypothesis, $|L/\theta |>n-2$, therefore $|L/\theta |=n-1$ and hence $[a,b]$ is a narrows, $b=a^{\prime }$, $L/\theta =\{\{a,b\}\}\cup \{\{x\}\ |\ x\in L\setminus \{a,b\}\}=\{\{a,a^{\prime }\}\}\cup \{\{x\}\ |\ x\in L\setminus \{a,a^{\prime }\}\}$, $|{\rm Con}(L/\theta )|\leq 2^{n-2}$ and $|{\rm Con}_{\BI }(L/\theta )|\leq 2^{\lfloor (n-1)/2\rfloor }$. If $|{\rm Con}(L/\theta )|=2^{n-2}$, then $L/\theta \cong {\cal L}_{n-1}$, hence $L\cong {\cal L}_n$, thus $|{\rm Con}_{\BI }(L)|=2^{\lfloor n/2\rfloor }$, which contradicts the current assumption. If $|{\rm Con}(L/\theta )|=2^{n-3}$, then $n\in 2(\N \setminus \{0,1\})+1$ and $L/\theta \cong {\cal L}_{(n-3)/2}\oplus {\cal L}_2^2\oplus {\cal L}_{(n-3)/2}$, hence $L/\theta \cong _{\BI}{\cal L}_{(n-3)/2}\oplus {\cal L}_2^2\oplus {\cal L}_{(n-3)/2}$ or $L/\theta \cong _{\BI}{\cal L}_{(n-3)/2}\oplus ({\cal L}_3\boxplus {\cal L}_3)\oplus {\cal L}_{(n-3)/2}$, the first of which can not hold for the current case of $\theta =eq(\{\{a,a^{\prime }\}\}\cup \{\{x\}\ |\ x\in L\setminus \{a,a^{\prime }\}\})$, and the latter of which gives us $L\cong _{\BI}{\cal L}_{(n-3)/2}\oplus N_5\oplus {\cal L}_{(n-3)/2}$, so that $|{\rm Con}_{\BI }(L)|=3\cdot 2^{\lfloor n/2-3\rfloor }$, which contradicts the current assumption. If $|{\rm Con}(L/\theta )|=5\cdot 2^{n-6}$, then $n\in 2(\N \setminus \{0,1,2\})$ and $L/\theta \cong {\cal L}_{n/2-2}\oplus N_5\oplus {\cal L}_{n/2-2}$, hence $L\cong _{\BI }{\cal L}_{n/2-2}\oplus ({\cal L}_4\boxplus {\cal L}_4)\oplus {\cal L}_{n/2-2}$, thus $|{\rm Con}_{\BI }(L)|=3\cdot 2^{\lfloor n/2-3\rfloor }$, which is another contradiction to the current assumption. The remaining case is $|{\rm Con}(L/\theta )|\leq 2^{n-4}<2^{n-3}=2^{(n-1)-2}$, but then $|{\rm Con}_{\BI }(L/\theta )|<2^{\lfloor (n-1)/2\rfloor }$ by the induction hypothesis, thus $2^{\lfloor n/2\rfloor }<|{\rm Con}_{\BI }(L)|<2^{\lfloor (n-1)/2\rfloor +1}=2^{\lfloor (n+1)/2\rfloor }$, so $\lfloor n/2\rfloor <\lfloor (n+1)/2\rfloor $, hence $n\in 2\N +1$, but then $2^{\lfloor (n-2)/2\rfloor }<|{\rm Con}_{\BI }(L/\theta )|\leq 2^{\lfloor (n-1)/2\rfloor }=2^{\lfloor (n-2)/2\rfloor }$ and we have a contradiction.

Now assume that $|{\rm Con}_{\BI }(L)|=2^{\lfloor n/2\rfloor }$, so that $|{\rm Con}_{\BI }(L/\theta )|\geq 2^{\lfloor n/2\rfloor }/2=2^{\lfloor n/2\rfloor -1}=2^{\lfloor (n-2)/2\rfloor }>2^{\lfloor (n-2)/2\rfloor -1}=2^{\lfloor (n-4)/2\rfloor }$, so that $|L/\theta |>n-4$ by the induction hypothesis, therefore $|L/\theta |\in \{n-1,n-2,n-3\}$.

{\bf Case 1:} $|L/\theta |=n-1$, hence $[a,b]$ is a narrows, $b=a^{\prime }$, $L/\theta =\{\{a,a^{\prime }\}\}\cup \{\{x\}\ |\ x\in L\setminus \{a,a^{\prime }\}\}$ and, by the induction hypothesis, $|{\rm Con}_{\BI }(L/\theta )|\leq 2^{\lfloor (n-1)/2\rfloor }$. We can apply the argument above, and obtain the possible situation $L\cong {\cal L}_n$, up until the last statement giving a contradiction, which under the current assumption becomes $n\in 2\N +1$ and $2^{\lfloor (n-2)/2\rfloor }\leq |{\rm Con}_{\BI }(L/\theta )|\leq 2^{\lfloor (n-1)/2\rfloor }=2^{\lfloor (n-2)/2\rfloor }$, so that $|{\rm Con}_{\BI }(L/\theta )|=2^{\lfloor (n-2)/2\rfloor }=2^{\lfloor (n-1)/2\rfloor }$, so, by the induction hypothesis, $L/\theta \cong {\cal L}_{n-1}$ or $n\in 2(\N \setminus \{0,1\})+1$ and $L/\theta \cong _{\BI }{\cal L}_{(n-3)/2}\oplus {\cal L}_2^2\oplus {\cal L}_{(n-3)/2}$, so that $L/\theta \cong {\cal L}_n$, since the latter situation can not hold for the current form of $\theta $.

{\bf Case 2:} $|L/\theta |=n-2$, so that:\begin{itemize}
\item $[a,b]$ is a narrows and: either $b=b^{\prime }$ and $L/\theta =\{\{a,b,a^{\prime }\}\}\cup \{\{x\}\ |\ x\in L\setminus \{a,b,a^{\prime }\}\}$, or $\{a,b\}\cap \{a^{\prime },b^{\prime }\}=\emptyset $ and $L/\theta =\{\{a,b\},\{a^{\prime },b^{\prime }\}\}\cup \{\{x\}\ |\ x\in L\setminus \{a,b,a^{\prime },b^{\prime }\}\}$;
\item or $[a,b]$ is not a narrows, $L/\theta =\{\{a,b\},\{a^{\prime },b^{\prime }\}\}\cup \{\{x\}\ |\ x\in L\setminus \{a,b,a^{\prime },b^{\prime }\}\}$, $\{a,b,b^{\prime },a^{\prime }\}\cong {\cal L}_2^2$ is a convex sublattice of $L$ and either $a=b\wedge b^{\prime }$ and $a/\theta \prec a^{\prime }/\theta $ or, dually, $b=a\vee a^{\prime }$ and $a^{\prime }/\theta \prec a/\theta $.\end{itemize}

By the induction hypothesis, $2^{\lfloor (n-2)/2\rfloor }\leq |{\rm Con}_{\BI }(L/\theta )|\leq 2^{\lfloor (n-2)/2\rfloor }$, so that $|{\rm Con}_{\BI }(L/\theta )|=2^{\lfloor (n-2)/2\rfloor }$, hence $L/\theta \cong {\cal L}_{n-2}$ or $n\in 2(\N \setminus \{0,1,2\})$ and $L/\theta \cong _{\BI }{\cal L}_{n/2-2}\oplus {\cal L}_2^2\oplus {\cal L}_{n/2-2}$. Then $L\cong {\cal L}_n$ or $n\in 2(\N \setminus \{0,1\})$ and $L\cong _{\BI }{\cal L}_{n/2-1}\oplus {\cal L}_2^2\oplus {\cal L}_{n/2-1}$ or $n\in 2(\N \setminus \{0,1,2\})$ and $L\cong _{\BI }{\cal L}_{n/2-2}\oplus B_6\oplus {\cal L}_{n/2-2}$; in the latter of these cases, $|{\rm Con}_{\BI }(L)|=5\cdot 2^{n/2-3}$, which contradicts the current assumption.

{\bf Case 3:} $|L/\theta |=n-3$, so that, if $a$ is meet--reducible, then $a\prec c$ for some $c\in L\setminus \{b\}$, $b\vee c=c^{\prime }$ and $a/\theta =\{a,b\}\prec c/\theta =\{c,b\vee c\}\prec a^{\prime }/\theta =\{a^{\prime },b^{\prime }\}$, $\{a,b,c,a^{\prime },b^{\prime },c^{\prime }\}\cong {\cal L}_2\times {\cal L}_3$, as in the sixth figure in Remark \ref{narrows}, and dually if $b$ is join--reducible. We have $2^{\lfloor (n-3)/2\rfloor }\leq 2^{\lfloor (n-2)/2\rfloor }\leq |{\rm Con}_{\BI }(L/\theta )|\leq 2^{\lfloor (n-3)/2\rfloor }$ by the induction hypothesis, hence $|{\rm Con}_{\BI }(L/\theta )|=2^{\lfloor (n-3)/2\rfloor }=2^{\lfloor (n-3)/2\rfloor }$, therefore $\lfloor (n-2)/2\rfloor =\lfloor (n-3)/2\rfloor $, that is $n\in 2\N +1$, and either $L/\theta \cong {\cal L}_{n-3}$ or $n\geq 7$ and $L/\theta _{\BI }\cong {\cal L}_{(n-5)/2}\oplus {\cal L}_2^2\oplus {\cal L}_{(n-5)/2}$. $n\in 2\N +1$, thus we can not have $L\cong {\cal L}_{\lfloor (n-4)/2\rfloor }\oplus ({\cal L}_2\times {\cal L}_3)\oplus {\cal L}_{\lfloor (n-4)/2\rfloor }$, and $L\cong {\cal L}_{(n-5)/2}\oplus ({\cal L}_3\boxplus ({\cal L}_2\times {\cal L}_3))\oplus {\cal L}_{(n-5)/2}$ would imply $L/\theta \cong _{\BI }{\cal L}_{(n-5)/2}\oplus ({\cal L}_3\boxplus {\cal L}_3)\oplus {\cal L}_{(n-5)/2}\ncong _{\BI }{\cal L}_{(n-5)/2}\oplus {\cal L}_2^2\oplus {\cal L}_{(n-5)/2}$, hence this case can not appear under the current assumption.

Now assume that $|{\rm Con}_{\BI }(L)|<2^{\lfloor n/2\rfloor }$ and $|{\rm Con}(L)|\geq 2^{n-2}$. Then $|{\rm Con}(L)|\in \{2^{n-1},2^{n-2}\}$, hence $L\cong {\cal L}_n$ or $n\in 2(\N \setminus \{0,1\})$ and $L\cong {\cal L}_{n/2-1}\oplus {\cal L}_2^2\oplus {\cal L}_{n/2-1}$, so that $L\cong _{\BI }{\cal L}_{n/2-1}\oplus {\cal L}_2^2\oplus {\cal L}_{n/2-1}$ or $L\cong _{\BI }{\cal L}_{n/2-1}\oplus ({\cal L}_3\boxplus {\cal L}_3)\oplus {\cal L}_{n/2-1}$, hence, by the current assumption, $L\cong _{\BI }{\cal L}_{n/2-1}\oplus ({\cal L}_3\boxplus {\cal L}_3)\oplus {\cal L}_{n/2-1}$, which is not a pseudo--Kleene algebra. If $|{\rm Con}(L)|=5\cdot 2^{n-5}$, then $n\in 2(\N \setminus \{0,1\})+1$ and $L\cong {\cal L}_{(n-3)/2}\oplus N_5\oplus {\cal L}_{(n-3)/2}$, which is neither modular, nor a pseudo--Kleene algebra, and has ${\rm Con}_{\BI }(L)\cong {\cal L}_2^{(n-5)/2}\times {\cal L}_3$, which is not a Boolean algebra. Hence, if $L$ is a pseudo--Kleene algebra or ${\rm Con}_{\BI }(L)$ is a Boolean algebra, so, in particular, if ${\rm Con}(L)$ is a Boolean algebra, in particular if $L$ is modular (see Proposition \ref{cglatbool} and Corollary \ref{casecglatbool}), then: $|{\rm Con}(L)|<2^{n-2}$ iff $|{\rm Con}(L)|\leq 2^{n-3}$. Finally, if ${\rm Con}_{\BI }(L)$ is a Boolean algebra, in particular if ${\rm Con}(L)$ is a Boolean algebra, in particular if $L$ is modular, then, clearly: $|{\rm Con}_{\BI }(L)|<2^{\lfloor n/2\rfloor }$ iff $|{\rm Con}_{\BI }(L)|\leq 2^{\lfloor n/2\rfloor -1}$.\end{proof}

\begin{corollary} If $L$ is a BZ--lattice with $0$ meet--irreducible and $|L|=n\in \N \setminus \{0,1\}$, then:\begin{itemize}
\item $|{\rm Con}_{\BZ }(L)|\leq 2^{\lfloor n/2\rfloor -1}+1$;
\item $|{\rm Con}_{\BZ }(L)|=2^{\lfloor n/2\rfloor -1}+1$ iff ${\rm Con}_{\BZ }(L)\cong {\cal L}_2^{\lfloor n/2\rfloor -1}\oplus {\cal L}_2$ iff $|{\rm Con}_{\BI }(L)|=2^{\lfloor n/2\rfloor }$ iff ${\rm Con}_{\BI }(L)\cong {\cal L}_2^{\lfloor n/2\rfloor }$ iff either $L\cong {\cal L}_n$ or $n\in 2(\N \setminus \{0,1,2\})$ and $L\cong {\cal L}_{n/2-1}\oplus {\cal L}_2^2\oplus {\cal L}_{n/2-1}$ iff $|{\rm Con}(L)|\in \{2^{n-1},2^{n-2}\}$ iff either ${\rm Con}(L)\cong {\cal L}_2^{n-1}$ or $n\geq 2$ and ${\rm Con}(L)\cong {\cal L}_2^{n-2}$;
\item $|{\rm Con}_{\BZ }(L)|<2^{\lfloor n/2\rfloor -1}+1$ iff $|{\rm Con}_{\BI }(L)|<2^{\lfloor n/2\rfloor }$ iff $|{\rm Con}(L)|<2^{n-2}$ iff $|{\rm Con}(L)|\leq 2^{n-3}$;
\item if ${\rm Con}_{\BZ }(L)\cong B\oplus {\cal L}_2$ for some Boolean algebra $B$, in particular if $L$ is modular, then: $|{\rm Con}_{\BZ }(L)|<2^{\lfloor n/2\rfloor -1}+1$ iff $|{\rm Con}_{\BZ }(L)|\leq 2^{\lfloor n/2\rfloor -2}+1$ iff $|{\rm Con}_{\BI }(L)|<2^{\lfloor n/2\rfloor }$ iff $|{\rm Con}_{\BI }(L)|\leq 2^{\lfloor n/2\rfloor -1}$ iff $|{\rm Con}(L)|<2^{n-2}$ iff $|{\rm Con}(L)|\leq 2^{n-3}$.\end{itemize}\label{maxcgaol}\end{corollary}

\begin{proof} By Lemma \ref{thecg} and Theorem \ref{maxcgaol}.\end{proof}

\begin{remark} Since, for all $n\in \N ^*$, ${\cal L}_n$ and ${\cal L}_{n/2-1}\oplus {\cal L}_2^2\oplus {\cal L}_{n/2-1}$, the latter for $n\in 2(\N \setminus \{0,1\})$, are, for instance, Kleene algebras, it follows from Theorem \ref{maxcgkl} that these are also exactly the $n$--element Kleene algebras with the most congruences. Since ${\cal L}_n$ and ${\cal L}_{n/2-1}\oplus {\cal L}_2^2\oplus {\cal L}_{n/2-1}$, the latter for $n\in 2(\N \setminus \{0,1,2\})$, are, for instance, distributive antiortholattices with $0$ meet--irreducible, it follows from Corollary \ref{maxcgaol} that these are also exactly the $n$--element distributive antiortholattices with $0$ meet--irreducible with the most congruences.

For any $n\in \N $ with $n\geq 5$, $2^{\lfloor n/2\rfloor -1}$ is the second largest number of congruences of an $n$--element modular i--lattice, as well as an $n$--element modular pseudo--Kleene algebra, and, for $n\geq 6$, even the second largest number of congruences of an $n$--element Kleene algebra, because, for instance:\begin{itemize}
\item if $n\in 2(\N \setminus \{0,1\})+1$, then $L={\cal L}_{(n-3)/2}\oplus M_3\oplus {\cal L}_{(n-3)/2}$ is an $n$--element modular pseudo--Kleene algebra with $|{\rm Con}_{\BI }(L)|=2^{(n-5)/2}\cdot 2=2^{(n-3)/2}=2^{(n-1)/2-1}=2^{\lfloor n/2\rfloor -1}$;
\item if $n\in 2(\N \setminus \{0,1,2\})$, then $L={\cal L}_{n/2-2}\oplus ({\cal L}_2\times {\cal L}_3)\oplus {\cal L}_{n/2-2}$ is an $n$--element Kleene algebra with $|{\rm Con}_{\BI }(L)|=2^{n/2-3}\cdot 2^2=2^{n/2-1}=2^{\lfloor n/2\rfloor -1}$;
\item if $n\in 2(\N \setminus \{0,1,2\})+1$, then $L={\cal L}_{(n-5)/2}\oplus {\cal L}_2^2\oplus {\cal L}_2^2\oplus {\cal L}_{(n-5)/2}$ is an $n$--element Kleene algebra with $|{\rm Con}_{\BI }(L)|=2^{(n-7)/2}\cdot 2^2=2^{(n-3)/2}=2^{(n-1)/2-1}=2^{\lfloor n/2\rfloor -1}$.\end{itemize}

Hence, for any $n\in \N $ with $n\geq 7$, $2^{\lfloor n/2\rfloor -2}+1$ is the second largest number of congruences of an $n$--element modular BZ--lattice with $0$ meet--irreducible and, for $n\geq 8$, even a distributive one.

Note, also, from Theorems \ref{maxcglat} and \ref{maxcgkl} and Corollary \ref{maxcgaol} that, if $L$ is a pseudo--Kleene algebra with $|L|=n\in \N ^*$ and $|{\rm Con}_{\BI }(L)|<2^{\lfloor n/2\rfloor }$, then $n\geq 5$, thus, if $L$ is a BZ--lattice with $0$ meet--irreducible, $|L|=n\in \N \setminus \{0,1\}$ and $|{\rm Con}_{\BZ }(L)|<2^{\lfloor n/2\rfloor -1}+1$, then $n\geq 7$.\end{remark}

\section*{Acknowledgements} This work was supported by the research grant {\em Propriet\`a d`Ordine Nella Semantica Algebrica delle Logiche Non--classiche} of Universit\`a degli Studi di Cagliari, Regione Autonoma della Sardegna, L. R. 7/2007, n. 7, 2015, CUP: F72F16002920002.

\end{document}